\documentclass[12pt]{article}

\usepackage{graphics,amsmath,amssymb,amsthm,mathrsfs}

\newcommand{\br}{\mathbb{R}}

\newcommand{\varep}{\varepsilon}
\newcommand{\brd}{\mathbb{R}^d}

\newtheorem{thm}{Theorem}[section]
\newtheorem{lemma}[thm]{Lemma}

\newtheorem{remark}[thm]{Remark}
\newcommand{\average}{-\!\!\!\!\!\!\int}

\numberwithin{equation}{section}

\begin{document}

\bibliographystyle{amsplain}

\title{Homogenization of Elliptic Systems\\
With Neumann Boundary Conditions}

\author{Carlos E. Kenig\thanks{Supported in part by NSF grant DMS-0968472}
 \and Fanghua Lin \thanks{Supported in part by NSF grant DMS-0700517}
\and Zhongwei Shen\thanks{Supported in part by NSF grant DMS-0855294}}

\date{ }

\maketitle

\begin{abstract}
For a family of second order elliptic systems with rapidly oscillating periodic
coefficients in a $C^{1,\alpha}$ domain, 
we establish uniform $W^{1,p}$ estimates, Lipschitz estimates, and nontangential
maximal function estimates on solutions with Neumann boundary conditions.
\end{abstract}

\section{Introduction and statement of main results}

The main purpose of this work is to study uniform regularity estimates
for a family of elliptic operators $\{\mathcal{L}_\varep, \varep>0\}$,
arising in the theory of homogenization, with rapidly oscillating periodic coefficients.
We establish sharp $W^{1,p}$ estimates, Lipschitz estimates, and nontangential
maximal function estimates, which are uniform in the parameter $\varep$,
on solutions with Neumann boundary conditions.

Specifically, we consider
\begin{equation}\label{operator}
\mathcal{L}_\varep
=-\frac{\partial}{\partial x_i}\left[
a_{ij}^{\alpha\beta}\left(\frac{x}{\varep}\right)
 \frac{\partial}{\partial x_j}\right]
=-\text{\rm div}\left[ A\left(\frac{x}{\varep}\right)\nabla \right],
\end{equation}
where $\varep>0$.
We assume that the coefficient matrix 
$A(y)=\big(a_{ij}^{\alpha\beta} (y)\big)$ with $1\le i,j\le d$ and $ \ 1\le \alpha, \beta\le m$ 
is real and satisfies the ellipticity condition 
\begin{equation}\label{ellipticity}
 \mu |\xi|^2 \le a_{ij}^{\alpha\beta} (y) \xi_i^\alpha \xi_j^\beta \le \frac{1}{\mu} |\xi|^2 
\quad \text{ for } y\in \brd \text{ and }  \xi=(\xi_i^\alpha)\in \mathbb{R}^{dm},
\end{equation}
where $\mu>0$, the periodicity condition
\begin{equation}\label{periodicity}
A(y+z)=A(y) \quad \text{ for } y\in \mathbb{R}^{d} \text{ and }
z\in \mathbb{Z}^{d},
\end{equation}
and the smoothness condition
\begin{equation}\label{smoothness}
| A(x)-A(y)| \le \tau |x-y|^\lambda
\quad \text{ for some } \lambda \in (0,1) \text{ and } \tau \ge 0.
\end{equation}
We will say $A\in \Lambda (\mu, \lambda,\tau)$ if $A=A(y)$ satisfies conditions 
(\ref{ellipticity}), (\ref{periodicity}) and (\ref{smoothness}).

Let $f\in L^2(\Omega)$ and $g\in W^{-1/2,2}(\partial\Omega)$.
Consider the Neumann boundary value problem 
\begin{equation}\label{Neumann-problem-1}
\left\{
\aligned
\mathcal{L}_\varep (u_\varep) & =\text{\rm div} (f)&   & \text{ in }\Omega,\\
\frac{\partial u_\varep}{\partial\nu_\varep} & = g-n\cdot f  & &  \text{ on } \partial\Omega,
\endaligned
\right.
\end{equation}
where
\begin{equation}\label{conormal}
\left( \frac{\partial u_\varep}{\partial\nu_\varep}\right)^\alpha
=n_i (x) a_{ij}^{\alpha\beta}\big(\frac{x}{\varep}\big)
\frac{\partial u^\beta_\varep}{\partial x_j}
\end{equation}
denotes the conormal derivative associated with $\mathcal{L}_\varep$ and
$n=(n_1, \dots, n_d)$ is the outward unit normal to $\partial\Omega$.
Assume that $\int_\Omega u_\varep =0$.
It is known from the theory of homogenization that 
under the assumptions (\ref{ellipticity})-(\ref{periodicity}),
$u_\varep \to u_0$ weakly in $W^{1,2}(\Omega)$ as $\varep\to 0$, where
 $\mathcal{L}_0 (u_0)=\text{\rm div}(f)$
in $\Omega$ and $\frac{\partial u_0}{\partial \nu_0}=g-n\cdot f 
$ on $\partial\Omega$.
Moreover, the homogenized operator $\mathcal{L}_0$ is an elliptic operator
with constant coefficients satisfying (\ref{ellipticity}) 
and depending only on the matrix $A$ (see e.g. \cite{bensoussan-1978}).

In this paper we shall be interested in sharp regularity estimates of $u_\varep$, 
which are uniform in the parameter $\varep$, 
assuming that the data are in $L^p$ or Besov or H\"older spaces.
The following three theorems are the main results of the paper.
Note that the symmetry condition $A^*=A$, i.e.,
\begin{equation}\label{symmetry}
a_{ij}^{\alpha\beta} (y) =a_{ji}^{\beta\alpha} (y)  \quad \text{ for } 1\le i,j\le d
\text{ and } 1\le \alpha, \beta \le m,
\end{equation}
is also imposed in Theorems \ref{Lipschitz-estimate-theorem} and \ref{maximal-function-theorem}.

\begin{thm}[{$W^{1,p}$ estimates}]
\label{W-1-p-theorem} 
Suppose $A\in \Lambda(\mu,\lambda,\tau)$ and $1<p<\infty$.
Let $\Omega$ be a bounded $C^{1,\alpha}$ domain for some $0<\alpha<1$.
Let $g=(g^\beta)\in B^{-1/p, p}(\partial\Omega)$,
$f=(f_j^\beta)\in L^p(\Omega)$ and $F=(F^\beta)\in L^q(\Omega)$,
where $q=\frac{pd}{p+d}$ for $p>\frac{d}{d-1}$ and $q>1$ for
$1<p\le \frac{d}{d-1}$.
Then, if $F$ and $g$
satisfy the compatibility condition
 $\int_\Omega F^\beta
 +<g^\beta,1>=0$ for $1\le \beta\le m$, the weak solutions 
to 
\begin{equation}\label{W-1-p}
\left\{ \aligned
\mathcal{L}_\varep (u_\varep) &=\text{\rm div} (f) +F
 & & \text{ in } \Omega,\\
\frac{\partial u_\varep}{\partial\nu_\varep} & =g-n\cdot f
 & & \text{ on } \partial\Omega,\\
u_\varep & \in W^{1, p}(\Omega)
\endaligned
\right.
\end{equation}
satisfy the estimate
\begin{equation}\label{W-1-p-estimate}
\| \nabla u_\varep\|_{L^p(\Omega)}
\le C\, \left\{  \| f\|_{L^p(\Omega)} + \|F\|_{L^q(\Omega)} 
+\| g\|_{B^{-1/p, p}(\partial\Omega)}\right\},
\end{equation}
where $C>0$ depends only on $d$, $m$, $p$, $q$,
 $\mu$, $\lambda$, $\tau$ and $\Omega$.
\end{thm}

\begin{thm}[Lipschitz estimates]
\label{Lipschitz-estimate-theorem}
Suppose that $A\in \Lambda(\mu,\lambda,\tau)$ and $A^*=A$.
Let $\Omega$ be a bounded $C^{1,\alpha}$ domain, $0<\eta<\alpha<1$ and $q>d$.
Then, for any $g\in C^{\eta} (\partial\Omega)$ and $F\in L^q(\Omega)$ with 
$\int_\Omega F +\int_{\partial\Omega} g=0$, 
 the weak solutions to 
\begin{equation}\label{Neumann-problem-3}
\left\{
\aligned
\mathcal{L}_\varep (u_\varep) & =F&   & \text{ in }\Omega,\\
\frac{\partial u_\varep}{\partial\nu_\varep} & = g  & &  \text{ on } \partial\Omega,\\
|\nabla u_\varep| & \in L^\infty(\Omega),
\endaligned
\right.
\end{equation}
satisfy the estimate
\begin{equation}\label{Lipschitz-estimate}
\|\nabla u_\varep\|_{L^\infty(\Omega)}
\le C \big\{  \| g\|_{C^{\eta}(\partial\Omega)} +\| F\|_{L^q(\Omega)}\big\},
\end{equation}
where $C>0$ depends only on $d$, $m$, 
$\eta$, $q$, $\mu$, $\lambda$, $\tau$ and $\Omega$.
\end{thm}

\begin{thm}[Nontangential maximal function estimates]
\label{maximal-function-theorem}
Suppose that $A\in \Lambda (\mu,\lambda,\tau)$ and $A=A^*$. 
Let $\Omega$ be a bounded $C^{1,\alpha}$ domain
and $1<p<\infty$.
Then, for any $g\in L^p(\partial\Omega)$ with mean value zero, 
 the weak solutions to
\begin{equation}\label{Neumann-problem-2}
\left\{
\aligned
\mathcal{L}_\varep (u_\varep) & =0&   & \text{ in }\Omega,\\
\frac{\partial u_\varep}{\partial\nu_\varep} & = g  & &  \text{ on } \partial\Omega,\\
(\nabla u_\varep)^*& \in L^p(\partial\Omega), & & 
\endaligned
\right.
\end{equation}
satisfy the estimate
\begin{equation}\label{maximal-function-estimate}
\|(\nabla u_\varep)^*\|_{L^p(\partial\Omega)}
+\| \nabla u_\varep\|_{L^q(\Omega)}\le C \, \| g\|_{L^p(\partial\Omega)},
\end{equation}
where $q=\frac{pd}{d-1}$ and
$C>0$ depends only on $d$, $m$, $p$, $\mu$, $\lambda$, $\tau$ and $\Omega$.
\end{thm}

A few remarks on notation are in order.
In Theorem \ref{W-1-p-theorem}, $B^{-1/p, p}(\partial\Omega)$ 
is the dual of the Besov space $B^{1/p, p^\prime}(\partial\Omega)$
on $\partial\Omega$, 
where $p^\prime=\frac{p}{p-1}$,
and $<g^\beta, 1>$ denotes the action of $g^\beta$ on the function $1$.
By a weak solution $u$ to (\ref{W-1-p}), we mean that $u\in W^{1,p}(\Omega)$ and satisfies
\begin{equation}\label{weak-formulation}
\int_\Omega a_{ij}^{\alpha\beta}
\left(\frac{x}{\varep}\right) \frac{\partial u_\varep^\beta}
{\partial x_j}
\cdot \frac{\partial \varphi^\alpha}{\partial x_i}\, dx
=\int_\Omega \left\{ -f_i^\alpha 
\frac{\partial \varphi^\alpha}{\partial x_i}
+F^\alpha\varphi^\alpha \right\}\, dx 
+<g^\alpha, \varphi^\alpha>,
\end{equation}
for any $\varphi =(\varphi^\alpha)\in C_0^1 (\br^d)$.
In Theorem \ref{maximal-function-theorem} we have used
$(\nabla u_\varep)^*$ to denote the nontangential maximal function of $\nabla u_\varep$.
We point out that the Lipschitz estimate in Theorem \ref{Lipschitz-estimate-theorem}
is sharp. Even with $C^\infty$ data, one cannot expect higher order uniform estimates
of $u_\varep$, as $\nabla u_\varep$ is known to converge to $\nabla u_0$ only weakly. 
As a result, the use of nontangential maximal functions in Theorem \ref{maximal-function-theorem}
to describe the sharp regularity of solutions with $L^p$ Neumann data appears to be natural and
necessary.
Also note that under the conditions (\ref{ellipticity}) and
(\ref{smoothness}), the existence and uniqueness (modulo additive 
constants) of solutions to (\ref{W-1-p}), (\ref{Neumann-problem-3})
and (\ref{Neumann-problem-2}) with sharp regularity
estimates are more or less well known 
(see e.g. \cite{Agmon-1959,Agmon-1964,Taylor-tools}).
What is new here is that with the additional periodicity assumption (\ref{periodicity}),
the constants $C$ in the regularity estimates (\ref{W-1-p-estimate}), (\ref{Lipschitz-estimate})
and (\ref{maximal-function-estimate}) are independent of
$\varep$. 

In the case of the Dirichlet boundary condition $u_\varep =g$ on $\partial\Omega$
with $g\in B^{1/p^\prime, p}(\partial\Omega)$ or $g\in C^{1,\eta}(\partial\Omega)$,
results analogous to Theorems \ref{W-1-p-theorem} and \ref{Lipschitz-estimate-theorem}
were established by Avellaneda and Lin in \cite{AL-1987,AL-1991}
for $C^{1,\alpha}$ domains (without the assumption $A^*=A$). They also obtained the
nontangential maximal function estimate $\|(u_\varep)^*\|_{L^p(\partial\Omega)}
\le C\| g\|_{L^p(\partial\Omega)}$ for solutions of $\mathcal{L}_\varep(u_\varep)=0$ in $\Omega$
(the case $m=1$ was given in \cite{AL-1987-ho}).
As it was noted in \cite{AL-1987}, uniform regularity estimates,
in addition to being of independent interest, have applications 
to homogenization of boundary control of distributed systems \cite{Lions-1985-IMA,
Lions-1988-SIAM, AL-1989-ho}.
Furthermore, they can be used to estimate convergence rates of $u_\varep\to u_0$ as
$\varep\to 0$. In particular, it was proved in \cite{AL-1987} that
$\| u_\varep-u_0\|_{L^\infty(\Omega)} =O(\varep)$,
if $\mathcal{L}_\varep (u_\varep)=\text{\rm div}(f)$ in $\Omega$,
$u_\varep =g$ on $\partial\Omega$, and $f,g$ are in certain function spaces.
Extending the Lipschitz estimate (\ref{Lipschitz-estimate})
to solutions with Neumann boundary conditions has been a longstanding
open problem.
The main reason why it is more difficult to deal with solutions 
with Neumann boundary
conditions in Theorem \ref{Lipschitz-estimate-theorem} than solutions with
Dirichlet boundary conditions in \cite{AL-1987,AL-1991}
is that now the boundary conditions in (\ref{Neumann-problem-3})
are $\varep$-dependent,
which causes new difficulties in the estimation of the appropriate boundary correctors.
We have overcome this difficulty, in the presence of symmetry,
thanks to the Rellich estimates obtained in \cite{Kenig-Shen-1, Kenig-Shen-2}.
Neumann boundary conditions are important in applications
of homogenization (see e.g. \cite{bensoussan-1978, Jikov-1994,Lions-1988-SIAM, Oleinik-1992}). 
The uniform estimates we establish in this paper can be used to study
convergence problems for solutions $u_\varep$, eigenfunctions and
eigenvalues with Neumann boundary conditions. 
As an example, 
let $w_\varep (x) =u_\varep (x)-u_0(x) -\varep \chi (\frac{x}{\varep}) \nabla u_0 (x)$,
where $\chi$ denotes the matrix of correctors for $\mathcal{L}_\varep$ in $\br^d$.
It can be shown that $w_\varep=w_\varep^{(1)} +w_\varep^{(2)}$, where
$\| \nabla w_\varep^{(1)}\|_{L^p(\Omega)} \le C_p \, \varep \|\nabla^2 u_0\|_{L^p(\Omega)}$
for any $1<p<\infty$, and $|\nabla w_\varep^{(2)} (x)|\text{\rm dist}(x, \partial\Omega)
\le C\varep \|\nabla u_0\|_{L^\infty(\partial\Omega)}$ for any $x\in \Omega$.
We will return to this in a forthcoming publication.

Let $N_\varep(x,y)$ denote the matrix of Neumann functions 
for $\mathcal{L}_\varep$
in $\Omega$ (see Section 5).
As a consequence of our uniform H\"older and Lipschitz estimates, we obtain the following 
bounds,
\begin{equation}\label{Neumann-function-bound-0}
\aligned
|N_\varep (x,y)| & \le \frac{C}{|x-y|^{d-2}},\\
|\nabla_x N_\varep (x,y)|+|\nabla_y N_\varep (x,y)| &\le \frac{C}{|x-y|^{d-1}},\\
|\nabla_x\nabla_y N_\varep (x,y)| &\le \frac{C}{|x-y|^d},
\endaligned
\end{equation}
for $d\ge 3$ (see Section 8).
In view of the work of Avellaneda and Lin on homogenization
of Poisson's kernel \cite{AL-1989-ho},
we remark that the techniques we develop in this paper
may also be used to establish asymptotics of
$ N_\varep (x,y)$.
This line of research,
 together with the convergence results mentioned above, 
will be developed in a forthcoming paper.

We should mention that the case $p=2$ in Theorem \ref{maximal-function-theorem}
is contained in \cite{Kenig-Shen-2}.
In fact, for the elliptic system $\mathcal{L}_\varep (u_\varep)=0$ 
in a bounded Lipschitz domain $\Omega$,
 the Neumann problem 
with the uniform estimate $\|(\nabla u_\varep)^*\|_{L^p(\partial\Omega)}
\le C\|\frac{\partial u_\varep}{\partial\nu_\varep}\|_{L^p(\partial\Omega)}$ 
and
the Dirichlet problem with the estimate $\|(u_\varep)^*\|_{L^p(\partial\Omega)}
\le C\| u_\varep\|_{L^p(\partial\Omega)}$, as well as the so-called regularity problem
with the estimate 
$\| (\nabla u_\varep)^*\|_{L^p(\partial\Omega)}\le C\| \nabla_{tan} u_\varep\|_{L^p(\partial\Omega)}$,
were solved recently by Kenig and Shen in \cite{Kenig-Shen-2} for $p$ close to $2$
(see \cite{Kenig-book} for references on boundary value problems
in Lipschitz domains for elliptic equations with constant coefficients).
The results in \cite{Kenig-Shen-2} are proved
under the assumption that $A\in \Lambda(\mu, \lambda, \tau)$ and $A^*=A$,
by the method of layer potentials.
In the case of a single equation ($m=1$), 
the $L^p$ solvabilities of Neumann, Dirichlet and regularity problems 
in Lipschitz domains with uniform nontangential
maximal function estimates were established in \cite{Kenig-Shen-1}
for the sharp ranges of $p$'s 
(the result for Dirichlet problem in Lipschitz domains
was obtained earlier by B. Dahlberg \cite{Dahlberg-personal}, 
using a different approach; see the appendix to \cite{Kenig-Shen-1} 
for Dahlberg's proof).
The results in \cite{Kenig-Shen-1, Kenig-Shen-2} rely on uniform
Rellich estimates $\|\frac{\partial u_\varep}{\partial\nu_\varep}\|_{L^2(\partial\Omega)}
\approx \|\nabla_{tan} u_\varep\|_{L^2(\partial\Omega)}$
for solutions of $\mathcal{L}_\varep (u_\varep)=0$ in a Lipschitz domain $\Omega$.
We point out that one of the key steps in the proof of Theorem \ref{Lipschitz-estimate-theorem}
uses the Rellich estimate $\|\nabla u_\varep\|_{L^2(\partial\Omega)}
\le C\|\frac{\partial u_\varep}{\partial\nu_\varep}\|_{L^2(\partial\Omega)}$
in a crucial way.

We now describe the key ideas in the proofs of our main results.
To show Theorem \ref{W-1-p-theorem}, we first establish the uniform 
boundary H\"older estimate for local solutions,
\begin{equation}\label{local-Holder-estimate-0}
\| u_\varep\|_{C^{0,\gamma}(B(Q, \rho)\cap\Omega)}
\le C \rho^{-\gamma} 
 \left(\average_{B(Q,2\rho)\cap\Omega}
 |u_\varep|^2 \, dx\right)^{1/2},
\end{equation}
for any $\gamma\in (0,1)$,
where $\mathcal{L}_\varep (u_\varep)=0$ in $B(Q,3\rho)\cap\Omega$
and $\frac{\partial u_\varep}{\partial\nu_\varep}=0$
on $B(Q,3\rho)\cap\partial\Omega$ for some $Q\in \partial\Omega$ and $0<\rho<c$.
The proof of (\ref{local-Holder-estimate-0}) uses a compactness method,
which was developed by Lin and Avellaneda
in \cite{AL-1987,AL-1989-II,AL-1989-ho} for homogenization problems,
with basic ideas originating from the 
regularity theory in the calculus of variations
and minimal surfaces.
As in the case of Dirichlet boundary condition, boundary correctors
are not needed for H\"older estimates with Neumann boundary condition.
From (\ref{local-Holder-estimate-0}) one may deduce the weak
reverse H\"older inequality,
\begin{equation}\label{reverse-Holder-0}
\left(\average_{B(Q, \rho)\cap\Omega} |\nabla u_\varep|^p\, dx\right)^{1/p}
\le C_p
\left(\average_{B(Q, 2\rho)\cap\Omega} |\nabla u_\varep|^2 \, dx\right)^{1/2}
\end{equation}
for any $p>2$. By \cite{Geng} this implies that $\|\nabla u_\varep\|_{L^p(\Omega)}
\le C\| f\|_{L^p(\Omega)}$ for $p>2$, if $\mathcal{L}_\varep (u_\varep)
=\text{\rm div}(f)$ in $\Omega$ and $\frac{\partial u_\varep}{\partial\nu_\varep}
=-n\cdot f$ on $\partial\Omega$.
The rest of Theorem \ref{W-1-p-theorem} follows by some duality arguments.

The proof of Theorem \ref{Lipschitz-estimate-theorem} is much more difficult than
that of Theorem \ref{W-1-p-theorem}. Assume that $0\in \partial\Omega$.
After a simple rescaling, the heart of matter here is to establish the uniform boundary
Lipschitz estimate for local solutions,
\begin{equation}\label{Lipschitz-estimate-0}
\|\nabla u_\varep\|_{L^\infty (B(0, 1)\cap\Omega)}
\le C\big\{  \| u_\varep\|_{L^\infty (B(0, 2)\cap \Omega)}
+\| g\|_{C^\eta(B(0,2)\cap\partial\Omega)}\big\},
\end{equation}
for some $\eta>0$, where $\mathcal{L}_\varep (u_\varep)=0$
in $B(0, 3)\cap \Omega$ and $\frac{\partial u_\varep}{\partial\nu_\varep}
=g$ on $B(0,3)\cap\partial\Omega$.
This problem  has been open for more than 20 years, ever since
the same estimate was established in \cite{AL-1987} for
local solutions with the Dirichlet boundary condition
$u_\varep=0$ in $B(0,3)\cap\partial\Omega$.
Our proof of (\ref{Lipschitz-estimate-0}) 
also uses the compactness method mentioned
above. However, as in the case of the Dirichlet boundary condition, 
one needs to introduce suitable boundary correctors in order to fully
take advantage of the fact that 
solutions of the homogenized system are in  $C^{1,\eta} (B(0,2)\cap \Omega)$.
A major technical breakthrough of this paper
 is the introduction and estimates of
such correctors $\Phi_\varep =(\Phi_{\varep, j}^{\alpha\beta})$, where
for each $1\le j\le d$ and $1\le \beta\le m$, $\Phi_{\varep, j}^\beta
=(\Phi_{\varep, j}^{1\beta}, \dots, \Phi_{\varep, j}^{m\beta})$ is the solution to
the Neumann problem
\begin{equation}\label{corrector-0}
\left\{
\begin{aligned}
\mathcal{L}_\varep (\Phi_{\varep, j}^\beta) & =0 &\qquad & \text{ in } \Omega,\\
\frac{\partial}{\partial \nu_\varep} \big( \Phi^\beta_{\varep, j}\big)
& =\frac{\partial }{\partial\nu_0} \big( P_j^\beta\big) & \qquad &\text{ on } \partial\Omega,\\
\Phi_{\varep, j}^\beta (0) & =0.
\end{aligned}
\right.
\end{equation}
Here $P_j^\beta =x_j (0, \cdots, 1, \dots, 0)$
with $1$ in the $\beta^{th}$ position
and $\frac{\partial w}{\partial \nu_0}$
denotes the conormal derivative of $w$ associated with the homogenized 
operator $\mathcal{L}_0$.
Note that by the boundary H\"older estimate,
$\Phi_{\varep, j}^{\alpha\beta} (x) \to x_j\delta_{\alpha\beta}$
uniformly in $\Omega$ as $\varep\to 0$. 
To carry out an elaborate compactness scheme in a similar fashion to that in \cite{AL-1987},
one needs to prove the uniform Lipschitz estimate for the solution of (\ref{corrector-0}),
\begin{equation}\label{corrector-Lipschitz}
\| \nabla \Phi_\varep\|_{L^\infty(\Omega)}\le C.
\end{equation}
The proof of (\ref{corrector-Lipschitz}) relies on two crucial observations. First, one can use
Rellich estimates as well as boundary H\"older estimates to show that
\begin{equation}\label{Neumann-estimate-0}
\int_{\partial\Omega} |\nabla_y \big\{
N_\varep (x,y)-N_\varep (z,y)\}|\, d\sigma (y) \le C,
\end{equation}
where $|x-z|\le c\,  \text{\rm dist}(x, \partial\Omega)$. 
Secondly, if $w_\varep(x)=\Phi_\varep (x) -x I -\varep \chi (x/\varep)$, then
$\frac{\partial w_\varep}{\partial \nu_\varep}$ can be represented as a sum of tangential
derivatives of $g_{ij}$ with $\|g_{ij}\|_{L^\infty(\partial\Omega)}
\le C\varep$.
Since $\mathcal{L}_\varep (w_\varep)=0$ in $\Omega$,
it follows from these observations as well as interior estimates that
$|\nabla w_\varep (x)|\le C\varep [\text{\rm dist} (x,\partial\Omega)]^{-1}$.
This gives the estimate $|\nabla \Phi_\varep (x)|\le C$, if
$\text{dist}(x, \partial\Omega)> \varep$.
The remaining case $\text{dist}(x,\partial\Omega)\le \varep$ follows
by a blow-up argument.
See Section 7 for details.
We note that the symmetry condition $A^*=A$ is only needed
for using the Rellich estimates.

With the Lipschitz estimate in Theorem \ref{Lipschitz-estimate-theorem}
at our disposal, Theorem \ref{maximal-function-theorem}
for $p>2$ follows from the case $p=2$ (established in \cite{Kenig-Shen-2} for
Lipschitz domains), by a real variable method originating in \cite{Caffarelli-1998} and
further developed in \cite{Shen-2005-bounds,
Shen-2006-ne,Shen-2007-boundary}.
The case $1<p<2$ is handled by establishing $L^1$ estimate for solutions with 
boundary data in the Hardy space $H^1(\partial\Omega)$ and then interpolating 
it with $L^2$ estimates,
as in the case of Laplacian \cite{Dahlberg-Kenig-1987} (see Section 9).
In view of the Lipschitz estimates in \cite{AL-1987}
for local solutions with Dirichlet boundary condition and
the $L^2$ estimates in \cite{Kenig-Shen-2},
a similar approach also solves the $L^p$ regularity 
problem with the estimate $\|(\nabla u_\varep)^*\|_{L^p(\partial\Omega)}
\le C\| \nabla_{tan} u_\varep\|_{L^p(\partial\Omega)}$ in a $C^{1,\alpha}$ domain $\Omega$ 
for all $1<p<\infty$ (see Section 10).
We further note that the same approach works
 equally well for the exterior domain
$\Omega_-=\br^d\setminus \overline{\Omega}$ and gives the solvabilities
of the $L^p$ Neumann and regularity problems in $\Omega_-$.
Consequently, as in the case of the Laplacian 
on a Lipschitz domain \cite{Verchota-1984,Dahlberg-Kenig-1987},
one may use the $L^p$ estimates in $\Omega$ and $\Omega_-$ and
the method of layer potentials to show that
solutions to the $L^p$ Neumann and regularity problems in $C^{1,\alpha}$
domains may be represented
by single layer potentials with density functions that are uniformly bounded 
in $L^p$.
Similarly, the solutions to the $L^p$ Dirichlet problem may be represented
by double layer potentials with uniformly $L^p$ bounded density functions
(see Section 11).
 
The summation convention will be used throughout the paper. 
Finally we remark that 
we shall make little effort to distinguish vector-valued functions or function spaces 
from their real-valued counterparts. 
This should be clear from the context.

\section{Homogenization and weak convergence}

Let $\mathcal{L}_\varep =-\text{\rm div}(A(x/\varep)\nabla)$ with
matrix $A(y)$ satisfying (\ref{ellipticity})-(\ref{periodicity}).
For each $1\le j\le d$ and $1\le\beta\le m$, let
$\chi_j^\beta =(\chi_j^{1\beta }, \dots, \chi_j^{m\beta})$ be the solution
of the following cell problem:
\begin{equation}\label{cell-problem}
\left\{
\aligned
& \mathcal{L}_1 (\chi_j^\beta)=-\mathcal{L}_1 (P^\beta_j) \quad \text{ in }\brd,\\
&\chi_j^\beta (y) \text{ is periodic with respect to }\mathbb{Z}^d,\\
& \int_{[0,1]^d}
\chi_j^\beta \, dy =0,
\endaligned
\right.
\end{equation}
where $P_j^\beta =P_j^\beta (y)
 =y_j(0, \dots, 1, \dots, 0)$ with $1$ in the $\beta^{th}$ position.
The matrix $\chi=\chi(y) =(\chi_j^{\alpha\beta}(y))$ with
$1\le j\le d$ and $1\le \alpha, \beta\le m$ is called the matrix of 
correctors for $\{ \mathcal{L}_\varep\}$.

With the summation convention
 the first equation in (\ref{cell-problem}) may be written
as
\begin{equation}\label{corrector-equation}
\frac{\partial}{\partial y_i}
\left[ a_{ij}^{\alpha\beta}
+a_{i\ell}^{\alpha\gamma}
\frac{\partial}{\partial y_\ell}\left( \chi_j^{\gamma\beta}\right)\right]=0
\quad
\text{ in }\brd.
\end{equation}
Let $\hat{A} =(\hat{a}_{ij}^{\alpha\beta})$, where $1\le i, j\le d$, $1\le \alpha, \beta\le m$ and
\begin{equation}
\label{homogenized-coefficient}
\hat{a}_{ij}^{\alpha\beta}
=\int_{[0,1]^d}
\left[ a_{ij}^{\alpha\beta}
+a_{i\ell}^{\alpha\gamma}
\frac{\partial}{\partial y_\ell}\left( \chi_j^{\gamma\beta}\right)\right]
\, dy.
\end{equation}
Then $\mathcal{L}_0=-\text{div}(\hat{A}\nabla)$ is the so-called
homogenized operator associated with 
$\{ \mathcal{L}_\varep\}$ (see \cite{bensoussan-1978}).
We need the following homogenization result.

\begin{lemma}
\label{lemma-2.1}
Let $\Omega$ be a bounded Lipschitz domain in $\brd$ and
$$
\text{\rm div}\left[A_k\left({x}/\varep_k\right)
\nabla u_k \right]=f\in W_0^{-1,2}(\Omega) \quad \text{ in } \Omega,
$$
where $\varep_k\to 0$ and the matrix $A_k(y)$ 
satisfies (\ref{ellipticity})-(\ref{periodicity}).
Suppose that $u_k \to u_0$ strongly in $L^2(\Omega)$,
$\nabla u_k\to \nabla u_0$ weakly in $L^2(\Omega)$ and
$A_k \big({x}/{\varep_k}\big) \nabla u_k$ converges weakly in $L^2(\Omega)$.
Also assume that the constant matrix $\hat{A_k}$, defined by (\ref{homogenized-coefficient})
(with $A$ replaced by $A_k$), converges to $A^0$.
Then 
$$
A_k \left({x}/{\varep_k}\right)\nabla u_k \to A^0 \nabla u_0 \
\text{ weakly in } L^2(\Omega)
$$
and $\text{\rm div}(A^0\nabla u_0)=f$ in $\Omega$.
\end{lemma}

\begin{proof}
If $A_k$ is independent of $k$, this is a classical result in the theory of
homogenization (see e.g. \cite{bensoussan-1978} or \cite{Chechkin-2007}).
The general case  may be proved by the same energy method. We give a proof here
for the sake of completeness.

Let $A_k =(a_{ij,k}^{\alpha\beta})$,
$\hat{A_k} =(\hat{a}_{ij,k}^{\alpha\beta})$ and
$A^0=(b_{ij}^{\alpha\beta})$.
Suppose that
\begin{equation}\label{definition-of-p}
a_{i\ell, k}^{\alpha\gamma} (x/\varep_k) \frac{\partial u_k^\gamma}{\partial x_\ell}
\to p_i^\alpha (x) \quad \text{ weakly in } L^2(\Omega).
\end{equation}
Clearly, $\text{\rm div}(P)=f$ in $\Omega$, where $P=(p_i^\alpha)$.
For $1\le j, \ell\le d$, $1\le \beta\le m$ and $k=1,2, \dots$, write
\begin{equation}\label{div-curl}
\aligned
& a_{i\ell, k}^{\alpha\gamma}
(x/\varep_k) \frac{\partial u_k^\gamma}{\partial x_\ell}
\cdot \frac{\partial }{\partial x_i}
\left\{ \varep_k \chi_{j,k}^{*\alpha\beta} (x/\varep_k)
+x_j \delta_{\alpha\beta}\right\}\\
& \qquad =\frac{\partial u_k^\gamma}{\partial x_\ell}
\cdot
a_{i\ell, k}^{\alpha\gamma}
\frac{\partial }{\partial x_i}
\left\{ \varep_k \chi_{j,k}^{*\alpha\beta} (x/\varep_k)
+x_j \delta_{\alpha\beta}\right\},
\endaligned
\end{equation}
where $\chi_k^* =(\chi_{j,k}^{*\alpha\beta})$ denotes the matrix of  correctors for
$(\mathcal{L}_\varep^k)^*$, the adjoint operator of $\mathcal{L}_\varep^k=-\text{div}
(A_k(x/\varep)\nabla )$.
By taking the weak limits on the both sides of (\ref{div-curl}) and
using a compensated compactness argument (see e.g. Lemma 5.1 in \cite{Chechkin-2007}),
we obtain
$$
\aligned
& p_i^\alpha (x) \cdot
\int_{[0,1]^d}
\left\{ \frac{\partial }{\partial y_i}
\left[ \chi_{j,k}^{*\alpha\beta} (y)\right]
+\delta_{ij}\delta_{\alpha\beta}\right\}\, dy\\
&\qquad
=\frac{\partial u_0^\gamma}{\partial x_\ell}
\cdot \lim_{k\to\infty}
\int_{[0,1]^d}
a_{i\ell, k}^{\alpha\gamma}
\left\{\frac{\partial}{\partial y_i}
\left[ \chi_{j,k}^{*\alpha\beta}(y)\right]
+\delta_{ij}\delta_{\alpha\beta}\right\}\, dy.
\endaligned
$$
Since
$$
\int_{[0,1]^d}
a_{i\ell, k}^{\alpha\gamma} (y)
\frac{\partial}{\partial y_i}
\left\{ \chi_{j,k}^{*\alpha\beta}(y)\right\} dy
=\int_{[0,1]^d}
a_{ji, k}^{\beta\alpha} (y)\frac{\partial}{\partial y_i}
\left\{ \chi_{\ell, k}^{\alpha\gamma} (y) \right\} dy
$$
(see e.g. \cite[p.122]{bensoussan-1978}),
it follows that
$$
\aligned
p_j^\beta (x)
&=\frac{\partial u_0^\gamma}{\partial x_\ell} \cdot
\lim_{k\to\infty}
\int_{[0,1]^d}
\left\{
a_{j\ell, k}^{\beta\gamma} (y)
+a_{ji, k}^{\beta\alpha}
\frac{\partial}{\partial y_i}
\big[ \chi_{\ell, k}^{\alpha\gamma}(y)\big]\right\} dy\\
&=\frac{\partial u_0^\gamma}{\partial x_\ell}\cdot
\lim_{k\to\infty} \hat{a}^{\beta\gamma}_{j\ell, k}\\
&=b_{j\ell}^{\beta\gamma} \cdot\frac{\partial u_0^\gamma}{\partial x_\ell}.
\endaligned
$$
In view of (\ref{definition-of-p}) this finishes the proof.
\end{proof}

Let $\psi:\mathbb{R}^{d-1}\to \mathbb{R}$ be a $C^{1,\alpha_0}$ function such that
\begin{equation}\label{psi}
\psi (0)=|\nabla\psi (0)|=0
\quad
\text{ and }\quad
\|\nabla \psi\|_{C^{\alpha_0}(\mathbb{R}^{d-1})}
\le M_0,
\end{equation}
where $\alpha_0\in (0,1)$ and $M_0>0$ will be fixed throughout the paper.
 For $r>0$, let
\begin{equation}
\label{definition-of-D}
\aligned
& D(r)=D(r, \psi) = \big\{ (x^\prime, x_d)\in\brd: \
|x^\prime|<r \text{ and } \psi(x^\prime)<x_d<\psi(x^\prime) +r \big\},\\
&\widetilde{D}(r)=\widetilde{D}(r, \psi) = \big\{ (x^\prime, x_d)\in\brd: \
|x^\prime|<r \text{ and } \psi(x^\prime)-r<x_d<\psi(x^\prime) +r \big\},\\
& \Delta (r)
=\Delta(r, \psi)  = 
\big\{ (x^\prime, \psi(x^\prime))\in\br^d: |x^\prime|<r \big\}.
\endaligned
\end{equation}

\begin{lemma}\label{lemma-2.2}
Let $\{\psi_k\}$ be a sequence of $C^{1,\alpha_0}$ functions satisfying (\ref{psi}).
Suppose that
$\psi_k \to \psi_0$ in $C^1(|x^\prime|<r)$ and
 $\{ \| v_k\|_{L^2(D(r, \psi_k))}\}$ is bounded.
Then there exist a subsequence, which we still denote by $\{ v_k\}$, and
$v_0\in L^2(D(r,\psi_0))$ such that
$v_k \to v_0 $ weakly in $
L^2(\Omega)$ for any $\Omega\subset\subset D(r, \psi_0)$.
\end{lemma}

\begin{proof}
Let $w_k(x^\prime, x_d)=v_k(x^\prime, x_d+\psi_k(x^\prime))$, defined
on 
$$
D(r,0)=\{ (x^\prime, x_d):\ 
|x^\prime|<r \text{ and } 0<x_d<r\}.
$$
Since $\{ w_k\}$ is bounded in $L^2(D(r, 0))$, there exists a subsequence, which
we still denote by $\{ w_k\}$, such that
$w_k \to w_0$ weakly in $L^2(D(r,0))$.
Let $v_0 (x^\prime, x_d)=w_0 (x^\prime, x_d-\psi_0(x^\prime))$.
It is not hard to verify that
$v_k\to v_0$ weakly in $L^2(\Omega)$ if $\Omega\subset\subset D(r, \psi_0)$.
\end{proof}

The following theorem plays an important role in our compactness argument for the
Neumann problem. Note that (\ref{Neumann-problem-k}) is the weak formulation of
$\text{\rm div}\big( A_k(x/\varep_k)\nabla u_k\big)=0$
in $D(r, \psi_k)$ and $\frac{\partial u_k}{\partial \nu_\varep^k}=g_k$
on $\Delta(r, \psi_k)$.

\begin{thm}\label{compactness-theorem}
Let $\{ A_k(y)\}$ be a sequence of matrices satisfying 
(\ref{ellipticity})-(\ref{periodicity}) and $\{ \psi_k\}$ a sequence of $C^{1,\alpha_0}$
functions satisfying (\ref{psi}).
Suppose that
\begin{equation}\label{Neumann-problem-k}
\int_{D(r, \psi_k)}
A_k(x/\varep_k)\nabla u_k \cdot \nabla \varphi\, dx
=\int_{\Delta(r, \psi_k)}
g_k \cdot \varphi\, d\sigma
\end{equation}
for any $\varphi\in C_0^1(\widetilde{D}(r, \psi_k))$, where
$\varep_k\to 0$ and
\begin{equation}\label{compactness-condition}
\| u_k \|_{W^{1,2}(D(r, \psi_k))}+
\| g_k \|_{L^2(\Delta (r, \psi_k))} \le C.
\end{equation}
Then there exist subsequences of $\{ \psi_k\}$, $\{ u_k\}$ and $\{ g_k\}$, which we still denote by
the same notation,  and a function $\psi_0$ satisfying (\ref{definition-of-p}),
$g_0\in L^2(\Delta(r,\psi_0))$, $u_0\in W^{1,2}(D(r, \psi_0))$, a constant
matrix $A^0$ such that
\begin{equation}\label{compactness-conclusion}
\left\{
\aligned
&\psi_k \to \psi_0 \text{ in } C^1(|x^\prime|<r),\\
& g_k(x^\prime, \psi_k(x^\prime)) \to g_0 (x^\prime, \psi_0 (x^\prime))
\quad \text{ weakly in } L^2 (|x^\prime|< r),\\
& u_k(x^\prime, x_d-\psi_k(x^\prime))
\to u_0 (x^\prime, x_d-\psi_0(x^\prime))
\quad \text{ strongly in } L^2(D(r,0)),
\endaligned
\right.
\end{equation}
and
\begin{equation}\label{compactness-conclusion-1}
\int_{D(r, \psi_0)}
A^0 \nabla u_0 \cdot \nabla \varphi\, dx
=\int_{\Delta(r,\psi_0)}
g_0 \cdot \varphi\, d\sigma
\end{equation}
for any $\varphi\in C_0^1(\widetilde{D}(r, \psi_0))$.
Moreover, the matrix $A^0$, as the limit of a subsequence
of $\{\hat{A}_k\}$,
 satisfies the condition
(\ref{ellipticity}).
\end{thm}

\begin{proof}
We first note that (\ref{compactness-conclusion})
follows directly from (\ref{compactness-condition}) by passing to subsequences.
To prove (\ref{compactness-conclusion-1}), we fix $\varphi\in C_0^1
(\widetilde{D}(r, \psi_0))$.
Clearly, if $k$ is sufficiently large, $\varphi\in C_0^1
(\widetilde{D}(r, \psi_k))$.
It is also easy to check that
$$
\int_{\Delta(r, \psi_k)}
g_k \cdot \varphi\, d\sigma
\to \int_{\Delta(r, \psi_0)} g_0 \cdot \varphi\, d\sigma.
$$
By passing to a subsequence we may assume that
$\hat{A}_k \to A^0$. 
Thus it suffices to show that
\begin{equation}\label{compactness-1}
\int_{D(r, \psi_k)}
A_k (x/\varep_k)\nabla u_k\cdot \nabla\varphi \, dx
\to 
\int_{D(r, \psi_0)}
A^0 \nabla u_0 \cdot \nabla\varphi \, dx.
\end{equation}
In view of Lemma \ref{lemma-2.2} we may assume that $\{ u_k\}$,
$\nabla u_k$, and $A_k(x/\varep_k)\nabla u_k$ converge weakly in $L^2(\Omega)$
for any $\Omega\subset\subset D(r,\psi_0)$.
As a result, $\{ u_k\}$ also converges strongly in $L^2(\Omega)$.

Now, given any $\delta>0$, we may choose a Lipschitz domain $\Omega$ such that
$\overline{\Omega}\subset D(r, \psi_0)$,
\begin{equation}\label{compactness-2}
\big|\int_{D(r,\psi_0)\setminus \Omega} 
A^0\nabla u_0\cdot \nabla \varphi\, dx\big|<\delta/3
\end{equation}
and
\begin{equation}\label{compactness-3}
\big|\int_{D(r,\psi_k)\setminus \Omega} 
{A}_k (x/\varep_k)\nabla u_k\cdot \nabla \varphi\, dx\big|<\delta/3
\end{equation}
for $k$ sufficiently large. Thus (\ref{compactness-1}) would follow if we can show that
\begin{equation}\label{compactness-4}
\int_\Omega
A_k (x/\varep_k) \nabla u_k \cdot \nabla \varphi\, dx
\to  
\int_\Omega
A^0  \nabla u_0 \cdot \nabla \varphi\, dx.
\end{equation}
This, however, is a direct consequence of Lemma \ref{lemma-2.1},
since
$\text{\rm div}(A_k(x/\varep_k)\nabla u_k)=0$ in $\Omega$
by (\ref{Neumann-problem-k}).
\end{proof}

We end this section with the uniform interior gradient estimate, established 
in \cite{AL-1987} by
Avellaneda and Lin, 
for solutions of $\mathcal{L}_\varep (u_\varep)=0$.
For a ball $B=B(x,r)$ in $\brd$,
we let $\rho B=B(x,\rho r)$.
We will use $\average_E f$ to denote $\frac{1}{|E|}\int_E f$, the average of $f$ over $E$.

\begin{thm}\label{interior-estimate-theorem}
Let $A \in \Lambda(\mu,\lambda,\tau)$.
Suppose that $\mathcal{L}_\varep (u_\varep)=0$ in $2B$.
Then
\begin{equation}\label{interior-estimate}
\sup_{B} |\nabla u_\varep|
\le C
\left(\average_{2B}
|\nabla u_\varep|^2\, dx \right)^{1/2},
\end{equation}
where $C$ depends only on $d$, $m$, $\mu$, $\lambda$, $\tau$.
\end{thm}

\section{Boundary H\"older estimates}

The goal of this section is to establish uniform
boundary H\"older estimates
for $\mathcal{L}_\varep$ under Neumann boundary condition.
Throughout this section we assume that 
$A\in \Lambda (\mu, \lambda,\tau)$.

\begin{thm}\label{boundary-holder-theorem}
Let $\Omega$ be a bounded $C^{1,\alpha_0}$ domain.
Let $p>0$ and $\gamma\in (0,1)$.
Suppose that $\mathcal{L}_\varep (u_\varep)=0$ in $B(Q, r)\cap\Omega$
and $\frac{\partial u_\varep}{\partial\nu_\varep} =g$
on $B(Q,r)\cap\partial\Omega$
for some $Q\in \partial\Omega$ and $0<r<r_0$. 
Then
\begin{equation}\label{local-size-estimate}
\sup_{B(Q, r/2)\cap\Omega} |u_\varep|
\le C\left\{
\left(
\average_{B(Q,r)\cap\Omega} |u_\varep|^p \, dx \right)^{1/p}
+\rho \| g\|_{L^\infty (B(Q,r)\cap\partial\Omega)}\right\},
\end{equation}
and for $x,y\in B(Q, r/2)\cap \Omega$,
\begin{equation}\label{local-holder-estimate}
|u_\varep (x)-u_\varep (y)|
\le
C\left(\frac{|x-y|}{r}\right)^\gamma
\left\{ \left(\average_{B(Q,r)\cap\Omega} |u_\varep|^p\, dx \right)^{1/p}
+\rho \| g\|_{L^\infty (B(Q,r)\cap\partial\Omega)}\right\},
\end{equation}
where $r_0>0$ depends only on $\Omega$ and $C>0$ 
on $d$, $m$, $\mu$, $\lambda$, $\tau$, $p$, $\gamma$ and $\Omega$.
\end{thm}

Let $D(\rho,\psi)$ and $\Delta (\rho,\psi)$ be defined by (\ref{definition-of-D}).
By a change of the coordinate system it will suffice
 to establish the following.

\begin{thm}\label{boundary-holder-theorem-local}
Let $\gamma\in (0,1)$.
Suppose that $\mathcal{L}_\varep (u_\varep)=0$ in $D(\rho)$ and
$\frac{\partial u_\varep}{\partial \nu_\varepsilon} =g$ on $\Delta(\rho)$
for some $\rho>0$.
Then for any $x,y\in D(\rho/2)$, 
\begin{equation}
\label{boundary-holder-estimate}
|u_\varep (x)-u_\varep (y)|
\le C \left(\frac{|x-y|}{\rho}\right)^\gamma 
\left\{
\left(\average_{D(\rho)}
|u_\varep |^2\right)^{1/2}
+\rho \| g\|_{L^\infty (\Delta (\rho))}\right\},
\end{equation}
where $D(\rho)=D(\rho, \psi)$, $\Delta(\rho)=D(\rho, \psi)$, and 
$C >0$ depends only on $d$, $m$, $\mu$, $\lambda$, 
$\tau$, $\gamma$ and $(\alpha_0, M_0)$ in (\ref{psi}).
\end{thm}

The proof of Theorem \ref{boundary-holder-theorem-local}
uses the compactness method developed in \cite{AL-1987, AL-1989-II,
AL-1989-ho} for homogenization problems.
We begin with the well known Cacciopoli's inequality,
\begin{equation}\label{Cacciopoli}
\int_{D(s\rho)}
|\nabla u_\varep|^2\, dx
\le \frac{C}{(t-s)^2 \rho^2}
\int_{D(t\rho)} |u_\varep|^2\, dx
+C \rho \| g\|_{L^2(\Delta(\rho))}^2,
\end{equation}
where $0<s<t<1$, $\mathcal{L}_\varep(u_\varep)=0$ in $D(\rho)$ and
$\frac{\partial u_\varep}{\partial\nu_\varep}=g$ on $\Delta(\rho)$.
The periodicity of $A$ is not needed here.

For a function $u$ defined on $S$, we will use $(\overline{u})_S$ 
(and $\average_S$) to denote its 
average over $S$. 

\begin{lemma}\label{Step-3.1-lemma}
Fix $\beta\in (0,1)$. There exist $\varep_0>0$ and $\theta\in (0,1)$, depending
only on $d$, $m$, $\mu$, $\lambda$, $\tau$, $\beta$ and $(\alpha_0, M_0)$, such that
\begin{equation}
\label{estimate-3.1}
\average_{D(\theta)}
|u_\varep -(\overline{u_\varep})_{D(\theta)}|^2
\le \theta^{2\beta},
\end{equation}
whenever $\varep<\varep_0$, $\mathcal{L}_\varep( u_\varep)=0$ in $D(1)$, 
$\frac{\partial u_\varep}{\partial \nu_\varep} =g$ on $\Delta (1)$,
$$
\| g\|_{L^\infty (\Delta (1))}\le 1 \quad \text{ and }\quad
\average_{D(1)}
|u_\varep -(\overline{u_\varep})_{D(1)} |^2
\le 1.
$$
\end{lemma}

\begin{proof}
Let $\mathcal{L}_0=-\text{\rm div} (A^0\nabla)$, where
$A^0$ is a constant matrix satisfying (\ref{ellipticity}).
Let $\beta^\prime =(1+\beta)/2$.
By boundary H\"older estimates  for solutions of elliptic systems with 
constant coefficients,
\begin{equation}
\label{constant-3.1}
\average_{D(r)}
|w - (\overline{w})_{D(r)}|^2
\le C_0 r^{2\beta^\prime}
\quad \quad \text{ for } 0<r<\frac{1}{4},
\end{equation}
whenever $\mathcal{L}_0 (w)=0$ in $D(1/2)$,
$\frac{\partial w}{\partial \nu_0}=g$ on $\Delta (1/2)$,
\begin{equation}\label{3.1.0}
 \| g\|_{L^\infty(\Delta(1/2))}\le 1 \qquad 
\text{ and } \qquad
\int_{D(1/2)} |w|^2 \le  |D(1)|,
\end{equation}
where  $C_0$ depends only on $d$, $m$, $\beta$,
$\mu$ and $(\alpha_0, M_0)$. 

Next we choose $\theta\in (0, 1/4)$
so small that $2C_0 \theta^{2\beta^\prime}\le \theta^{2\beta}$.
We shall show by contradiction that for this $\theta$,
there exists $\varep_0>0$, depending only
on $d$, $m$, $\mu$, $\lambda$, $\tau$, $\beta$ and $(\alpha_0, M_0)$, such that 
(\ref{estimate-3.1}) holds if $0<\varep<\varep_0$ and
 $u_\varep$ satisfies the conditions in 
Lemma \ref{Step-3.1-lemma}.

To this end let's suppose that there exist sequences $\{ \varep_k\}$, 
$\{ A_k \}$,
$\{ u_{\varep_k}\}$, $\{ g_k\}$ and $\{ \psi_k\}$ such that
$\varep_k \to 0$, $A_k \in \Lambda(\mu, \lambda, \tau)$,
$\psi_k$ satisfies (\ref{psi}),
\begin{equation}\label{3.1.1}
\left\{
\aligned
\mathcal{L}^k_{\varep_k} (u_{\varep_k}) & =0 & & \text{ in } D_k (1),\\
\frac{\partial u_{\varep_k}}{\partial\nu_{\varep_k}} & =g_k
& & \text{ on }\Delta_k(1),
\endaligned
\right.
\end{equation}
\begin{equation}\label{3.1.2}
\| g_k \|_{L^\infty(\Delta_k (1))} \le 1, \qquad
\average_{D_k(1)}
|u_{\varep_k} -(\overline{u_{\varep_k}})_{D_k(1)}|^2\le 1
\end{equation}
and
\begin{equation}\label{3.1.3}
\average_{D_k (\theta)}
|u_{\varep_k} -(\overline{u_{\varep_k}})_{D_k(\theta)}|^2
> \theta^{2\beta},
\end{equation}
where $\mathcal{L}_{\varep_k}^k
=-\text{\rm div} \big(A_k(x/\varep_k)\nabla\big)$,
$D_k(r) =D(r, \psi_k)$ and $\Delta_k (r)=D(r, \psi_k)$.
By subtracting a constant we may assume that $(\overline{u_{\varep_k}})_{D_k(1)} =0$.
Thus it follows from (\ref{3.1.2}) and the Cacciopoli's inequality 
(\ref{Cacciopoli}) that the norm of
$ u_{\varep_k}$ in $W^{1,2}(D_k(1/2)) $ is uniformly bounded.
In view of Theorem \ref{compactness-theorem}, by passing to subsequences, we may assume that
\begin{equation}\label{3.1.4}
\left\{
\aligned
&\psi_k \to \psi_0  \quad \text{ in }C^1(|x^\prime|<1),\\
&g_k(x^\prime, \psi_k (x^\prime))\to g_0(x^\prime, \psi_0(x^\prime)) \quad\text{ weakly
in } L^2(|x^\prime|<1),\\
& u_{\varep_k}(x^\prime, x_d-\psi_k (x^\prime))\to u_0 (x^\prime, x_d-\psi_0(x^\prime))
\quad \text{ strongly in } L^2(D(1/2,0)),
\endaligned
\right.
\end{equation}
and
\begin{equation}\label{3.1.5}
\left\{
\aligned
& \text{div}(A^0\nabla u_0)=0 & \quad & \text{ in } D(1/2, \psi_0),\\
& \frac{\partial u_0}{\partial \nu_0} =g_0 & \quad&  \text{ on } \Delta(1/2, \psi_0),
\endaligned
\right.
\end{equation}
where $A^0$ is a constant matrix satisfying (\ref{ellipticity}).

Using (\ref{3.1.4}) one may verify that 
$$
|D_k(r)|\to |D_0(r)|,\ \
\| g_0\|_{L^\infty(\Delta(1, \psi_0))}\le 1,\ \
(\overline{u_{\varep_k}})_{D_k(r)}\to (\overline{u_0})_{D_0(r)}
$$ 
and
\begin{equation}\label{3.1.6}
\int_{D_k(r)}
|u_{\varep_k}
-(\overline{u_{\varep_k}})_{D_k(r)}|^2
\to 
\int_{D_0(r)}
|u_0
-(\overline{u_0})_{D_0(r)}|^2
\end{equation}
for any $r\in (0,1]$, where $D_0(r)=D(r, \psi_0)$.
It follows that
\begin{equation}\label{3.1.6-1}
\aligned
\average_{D_0(1)} |u_0|^2  & \le 1,\\
\average_{D_0 (\theta)}
|u_{0} -(\overline{u_{0}})_{D_0(\theta)}|^2
& \ge \theta^{2\beta}.
\endaligned
\end{equation}
In view of (\ref{constant-3.1})-(\ref{3.1.0}) and (\ref{3.1.6-1}) we obtain
$\theta^{2\beta}\le C_0 \theta^{2\beta^\prime}$.
This contradicts $2C_0\theta^{2\beta^\prime}
\le \theta^{2\beta}$.
\end{proof}

\begin{lemma}\label{Step-3.2-lemma}
Fix $\beta\in (0,1)$. Let $\varep_0$, $\theta$ be the constants 
given by Lemma \ref{Step-3.1-lemma}.
Suppose that $\mathcal{L}_\varep (u_\varep) =0$
in $D(1, \psi)$ and $\frac{\partial u_\varep}{\partial\nu_\varep}=g$ on $\Delta(1,\psi)$.
Then, if $\varep< \theta^{k-1}\varep_0$ for some $k\ge 1$,
\begin{equation}\label{3.2.1}
\average_{D(\theta^k, \psi)}
|u_\varep -(\overline{u_\varep})_{D(\theta^k, \psi)}|^2
\le \theta^{2k\beta} J^2,
\end{equation}
where
$$
J=\max \left\{ 
\left(\average_{D(1, \psi)}
|u_\varep -(\overline{u_\varep})_{D(1, \psi)}|^2\right)^{1/2},\
\| g\|_{L^\infty(\Delta(1,\psi))}\right\}.
$$
\end{lemma}

\begin{proof}
The lemma is proved by induction on $k$.
Note that the case $k=1$ is given by Lemma \ref{Step-3.1-lemma}.
Assume now that the lemma holds for some $k\ge 1$.
Let $\varep < \theta^k \varep_0$. We apply Lemma \ref{Step-3.1-lemma}
to $w(x)=u(\theta^k x)$ in $D(1, \psi_k)$, where
$\psi_k(x)=\theta^{-k}\psi(\theta^k x)$.
Since $\mathcal{L}_{\varep/\theta^k} (w)=0$ in $D(1,\psi_k)$,
this gives
$$
\aligned
& 
\average_{D(\theta^{k+1}, \psi)}
|u_\varep- (\overline{u_\varep})_{D(\theta^{k+1}, \psi)}|^2\\
& \quad =
\average_{D(\theta, \psi_k)}
|w- (\overline{w})_{D(\theta, \psi_k)}|^2\\
&\quad
\le \theta^{2\beta}
\max \left\{\average_{D(1, \psi_k)}
|w-(\overline{w})_{D(1, \psi_k)}|^2,\
\theta^{2k} \| g\|^2_\infty\right\}\\
&\quad
= \theta^{2\beta}
\max \left\{\average_{D(\theta^k, \psi)}
|u_\varep-(\overline{u_\varep})_{D(\theta^k, \psi)}|^2,\
\theta^{2k} \| g\|^2_\infty\right\}\\
&\quad
\le \theta^{2(k+1)\beta} J^2,
\endaligned
$$
where $\| g\|_\infty
=\| g\|_{L^\infty(\Delta(1,\psi))}$ and the last step follows by the 
induction assumption.
Here we also have used the fact that 
$\|\nabla\psi_k\|_{C^{\alpha_0}(\mathbb{R}^{d-1})}
\le \|\nabla \psi\|_{C^{\alpha_0}(\mathbb{R}^{d-1})} \le M_0$.
\end{proof}

\noindent{\bf Proof of Theorem \ref{boundary-holder-theorem-local}}.
By rescaling we may assume that $\rho=1$. We may also assume that
$\varep< \varep_0$, since the case $\varep\ge\varep_0$ follows directly from the 
classical regularity theory. We may further assume that
$$
\| g\|_{L^\infty(\Delta(1))}\le 1
\quad\text{ and } \quad
\int_{D(1)} |u_\varep|^2\le 1.
$$
Under these assumptions we will show that
\begin{equation}\label{3.3.1}
\average_{D(r)}
|u_\varep-(\overline{u_\varep})_{D(r)} |^2 \le C r^{2\beta}
\end{equation}
for any $r\in (0,1/4)$.
The desired estimate (\ref{boundary-holder-estimate}) with $p=2$ follows from 
the interior estimate (\ref{interior-estimate}) and (\ref{3.3.1}),
using Campanato's characterization of H\"older spaces (see e.g. \cite{Giaquinta}).

To prove (\ref{3.3.1}) we first consider the case $r\ge (\varep/\varep_0)$.
Choose $k\ge 0$ so that $\theta^{k+1}\le r< \theta^{k}$.
Then $\varep \le \varep_0 r <\varep_0\theta^k$.
It follows from Lemma \ref{Step-3.2-lemma} that
$$
\aligned
& \average_{D(r)}
|u_\varep -(\overline{u_\varep})_{D(r)}|^2
\le  C\average_{D(\theta^k)}
|u_\varep -(\overline{u_\varep})_{D(\theta^k)}|^2
\\
&\quad\quad
\le C\theta^{2k\beta}
\le Cr^{2\beta}.
\endaligned
$$

Next suppose that $r<(\varep/\varep_0)$.
Let $w(x)=u_\varep (\varep x)$. Then
$\mathcal{L}_1 (w)=0$ in $D(\varep_0^{-1}, \psi_\varep)$, where
$\psi_\varep (x^\prime)=\varep^{-1} \psi(\varep x^\prime)$.
By the classical regularity we obtain
$$
\aligned
&\average_{D(r,\psi)} |u_\varep -(\overline{u_\varep})_{D(r,\psi)}|^2
=\average_{D(\frac{r}{\varep}, \psi_\varep)}
|w-(\overline{w})_{D(\frac{r}{\varep}, \psi_\varep)}|^2\\
&\le C\left(\frac{r}{\varep}\right)^{2\beta}
\max \left\{ 
\average_{D(\frac{1}{\varep_0}, \psi_\varep)}
|w-(\overline{w})_{D(\frac{1}{\varep_0}, \psi_\varep)}|^2,\
\varep^2 \| g\|_\infty\right\}\\
& = C\left(\frac{r}{\varep}\right)^{2\beta}
\max \left\{ 
\average_{D(\frac{\varep}{\varep_0}, \psi)}
|u_\varep-(\overline{u_\varep})_{D(\frac{\varep}{\varep_0}, \psi)}|^2,\
\varep^2 \| g\|_\infty\right\}\\
&\le  C\left(\frac{r}{\varep}\right)^{2\beta}
\left(\frac{\varep}{\varep_0}\right)^{2\beta}
 =C\varep_0^{-2\beta} r^{2\beta},
\endaligned
$$
where the last inequality follows from the previous case $r=(\varep/\varep_0)$.
This finishes the proof of (\ref{3.3.1}) and thus of Theorem \ref{boundary-holder-theorem-local}.
\qed

We are now in a position to give the proof of Theorem \ref{boundary-holder-theorem}.

\noindent{\bf Proof of Theorem \ref{boundary-holder-theorem}.}
By rescaling we may assume that $r=1$.
The case $p=2$ follows directly from Theorem \ref{boundary-holder-theorem-local}.
To handle the case $0<p<2$, we note that by a simple covering argument,
estimate (\ref{local-size-estimate}) for $p=2$ gives
\begin{equation}\label{3.3.2}
\sup_{B(Q,s)\cap \Omega} |u_\varep|
\le 
C\left\{ \frac{1}{(t-s)^d}
\left(\average_{B(Q,t)\cap\Omega} |u_\varep|^2\right)^{1/2}
+ \| g\|_{L^\infty(B(Q, 1)\cap\partial\Omega)}\right\},
\end{equation}
where $(1/4)<s<t<1$. By a convexity argument (see e.g. \cite[p.173]{Fefferman-Stein-1972}),
estimate (\ref{3.3.2})  
implies that for any $p>0$,
\begin{equation}\label{3.3.3}
\left(\average_{B(Q,1/2)\cap\Omega} |u_\varep|^2\right)^{1/2}
\le 
C_p\left\{
\left(\average_{B(Q,1)\cap\Omega} |u_\varep|^p\right)^{1/p}
+ \| g\|_{L^\infty(B(Q, 1)\cap\partial\Omega)}\right\}.
\end{equation}
The case $0<p<2$ now follows from estimate (\ref{3.3.3}) and the case $p=2$.
\qed

\section{Proof of Theorem \ref{W-1-p-theorem}}\label{section-4}

Under conditions (\ref{ellipticity}) and (\ref{smoothness}),
weak solutions to (\ref{W-1-p}) exist and are unique, up to
an additive constant, provided that the data
satisfy the necessary condition
$\int_\Omega F^\beta+<g^\beta, 1>=0$ for $1\le \beta\le m$.
In this section we will show that the weak solutions satisfy 
the uniform $W^{1,p}$ estimate in Theorem \ref{W-1-p-theorem}.

Our starting point is the following
theorem established by J. Geng in \cite{Geng}, using a real variable
method originating in \cite{Caffarelli-1998} and further developed
in \cite{Shen-2005-bounds, Shen-2006-ne, Shen-2007-boundary}.

\begin{thm}\label{Geng-theorem}
Let $p>2$ and $\Omega$ be a bounded Lipschitz domain. Let $\mathcal{L}=-\text{\rm div}(A (x)\nabla)$
be an elliptic operator with coefficients satisfying (\ref{ellipticity}).
Suppose that 
\begin{equation}\label{reverse-holder}
\left\{ \average_{B\cap\Omega}
|\nabla u|^p\right\}^{1/p}
\le C_0 
\left\{ \average_{2B\cap\Omega}
|\nabla u|^2\right\}^{1/2},
\end{equation}
whenever $u\in W^{1,2} (3B\cap\Omega)$, 
$\mathcal{L} (u)  =0$ in $3B\cap\Omega$,
and $\frac{\partial u}{\partial\nu}  =0$ on $3B\cap \partial\Omega$.
Here $B=B(Q, r)$ is a ball with the property
that $0<r<r_0$ and either $Q\in \partial\Omega$ or
$B(Q,3r)\subset \Omega$.
Then, for any $f\in L^p (\Omega)$, the unique (up to constants) $W^{1,2}$ solution to
\begin{equation}
\left\{
\aligned
 \mathcal{L} (u) & =\text{\rm div} (f)  & \quad & \text{ in } \Omega,\\
\frac{\partial u}{\partial\nu}  
&=-n\cdot f &\quad  & \text{ on } \partial\Omega,
\endaligned
\right.
\end{equation}
satisfies the estimate
\begin{equation}\label{4.1.2}
\| \nabla u\|_{L^p(\Omega)} \le C_p \| f\|_{L^p(\Omega)},
\end{equation}
where $C_p$ depends only on $d$, $m$, $p$,
$\mu$, $r_0$, $\Omega$ and the constant $C_0$ in (\ref{reverse-holder}).
\end{thm}

Now,  given $A\in \Lambda (\mu, \lambda, \tau)$ and $p>2$.
Let $\Omega$ be a $C^{1,\alpha_0}$ domain.
Suppose that $\mathcal{L}_\varep (u_\varep)=0$ in $3B\cap\Omega$ and
$\frac{\partial u_\varep}{\partial_{\nu_\varep}}=0$
on $3B\cap\partial\Omega$.
If $3B\subset \Omega$, the weak reverse H\"older 
inequality (\ref{reverse-holder}) for $u_\varep$ follows
from the interior estimate (\ref{interior-estimate}).
Suppose that $Q\in \partial\Omega$ and $B=B(Q,r)$.
We may use the interior estimate 
and boundary H\"older estimate (\ref{local-holder-estimate})
to obtain
\begin{equation}\label{4.1.3}
\aligned
|\nabla u_\varep (x)|
&\le C \delta(x)^{-1}\left(\average_{B(x,c\delta (x))}
|u_\varep (y)- u_\varep (x)|^2\, dy\right)^{1/2}\\
&\le C_\gamma
\left(\frac{r}{\delta (x)}\right)^{\gamma}
\left(\average_{B(Q,2r)\cap \Omega}
|\nabla u_\varep|^2\, dy\right)^{1/2}
\endaligned
\end{equation}
for any $\gamma\in (0,1)$ and $x\in B(Q,r)\cap \Omega$,
where $\delta (x)=\text{\rm dist}(x, \partial\Omega)$.
Choose $\gamma\in (0,1)$ so that $p\gamma<1$.
It is easy to see that (\ref{4.1.3}) implies
$$
\left(\average_{B\cap\Omega}
|\nabla u_\varep|^p\right)^{1/p}
\le C_p
\left(\average_{2B\cap\Omega}
|\nabla u_\varep|^2\right)^{1/2}.
$$
In view of Theorem \ref{Geng-theorem} 
we have proved Theorem \ref{W-1-p-theorem}
for the case $p>2$, $g=0$ and $F=0$.

\begin{lemma}\label{lemma-4.1}
Suppose $A\in \Lambda(\mu, \lambda,\tau)$.
Let $f\in L^p(\Omega)$, where
$\Omega$ be a bounded $C^{1,\alpha_0}$ domain and $1<p<\infty$.
Let $u\in W^{1,p}(\Omega)$ be a weak solution 
to $\mathcal{L}_\varep (u_\varep)=\text{\rm div}(f)$ in $\Omega$ and
$\frac{\partial u_\varep}{\partial\nu_\varep} =-n\cdot f$ on $\partial\Omega$.
Then $\|\nabla u_\varep\|_{L^p(\Omega)}
\le C_p\, \| f\|_{L^p(\Omega)}$.
\end{lemma}

\begin{proof}
The case $p>2$ was proved above. 
Suppose that $1<p<2$.
Let $g\in C^\infty_0(\Omega)$ and $v_\varep$ be a weak solution of 
$\mathcal{L}_\varep^* (v_\varep)=\text{div} (g)$ 
and $\frac{\partial v_\varep}{\partial\nu^*_{\varep}}
=0$ on $\partial\Omega$,
where $\mathcal{L}_\varep^*$ denotes the adjoint of $\mathcal{L}_\varep$.
Since $A^*\in \Lambda(\lambda, \mu, \tau)$ and $p^\prime>2$,
we have $\|\nabla v_\varep\|_{L^{p^\prime}(\Omega)} \le C \| g\|_{L^{p^\prime}(\Omega)}$. 
Also, note that 
\begin{equation}\label{4.1.1}
\int_\Omega
f_i^\alpha \cdot \frac{\partial v_\varep^\alpha}{\partial x_i}\, dx
=\int_\Omega a_{ij}^{\alpha\beta} \left(\frac{x}{\varep}\right)
\frac{\partial u_\varep^\beta}{\partial x_j} \cdot\frac{\partial v_\varep^\alpha}{\partial x_i}\, dx
=\int_\Omega g_i^\alpha \cdot \frac{\partial u_\varep^\alpha}{\partial x_i}\,dx,
\end{equation}
where $f=(f_i^\alpha)$ and $g=(g_i^\alpha)$.
The estimate $\|\nabla u_\varep\|_{L^p(\Omega)}
\le C\| f\|_{L^p(\Omega)}$ now follows from (\ref{4.1.1}) by duality.
\end{proof}

\begin{lemma}\label{lemma-4.2}
Suppose that $A\in \Lambda(\lambda, \mu, \tau)$.
Let $g=(g^\alpha)\in B^{-1/p, p}(\partial\Omega)$, where
$\Omega$ is a bounded $C^{1, \alpha_0}$ domain, $1<p<\infty$ and
$<g^\alpha,1>=0$. Let $u\in W^{1, p}(\Omega)$ be a weak solution to
$\mathcal{L}_\varep (u_\varep)=0$ in $\Omega$
and $\frac{\partial u_\varep}{\partial \nu_\varep}
=g$ on $\partial\Omega$.
Then $\|\nabla u_\varep\|_{L^p(\Omega)}\le C_p\, \| g\|_{B^{-1/p, p}(\partial\Omega)}$.
\end{lemma}

\begin{proof}
Let $f\in C_0^\infty(\Omega)$ and $v_\varep$ be a weak solution to
$\mathcal{L}_\varep^* (v_\varep)=\text{\rm div}(f)$ in $\Omega$ and $\frac{\partial v_\varep}
{\partial\nu^*_\varep}=0$ on $\partial\Omega$.
Since $A^*\in \Lambda(\lambda, \mu,\tau)$, by Lemma \ref{lemma-4.1},
we have $\|\nabla v_\varep\|_{L^{p^\prime}(\Omega)} \le C\, \| f\|_{L^{p^\prime}(\Omega)}$.

Note that
\begin{equation}\label{4.2.1}
\int_\Omega f_i^\alpha \cdot \frac{\partial u_\varep^\alpha}{\partial x_i}\, dx
=-\int_\Omega a_{ij}^{\alpha\beta}
\left(\frac{x}{\varep}\right) \frac{\partial u_\varep^\beta}{\partial x_j}
\cdot \frac{\partial v_\varep^\alpha}{\partial x_i}\, dx
=-<g, v_\varep>.
\end{equation}
Let $E$ be the average of $v_\varep$ over $\Omega$. Then
\begin{equation}\label{4.2.2}
\aligned
\big| <g, v_\varep>\big|
& =\big| <g, v_\varep-E>\big|
\le \| g\|_{B^{-1/p, p}(\partial\Omega)} 
\| v_\varep -E\|_{B^{1/p, p^\prime}(\partial\Omega)}\\
&\le C\, \| g\|_{B^{-1/p, p}(\partial\Omega)} \| v_\varep-E\|_{W^{1,p^\prime}(\Omega)} \\
& \le C\, \| g\|_{B^{-1/p, p}(\partial\Omega)} 
\|\nabla v_\varep\|_{L^{p^\prime}(\Omega)}\\
& \le C
\| g\|_{B^{-1/p, p}(\partial\Omega)}  \| f\|_{L^{p^\prime} (\Omega)},
\endaligned
\end{equation}
where we have used a trace theorem for the second inequality
and Poincar\'e inequality for the third.
The estimate $\|\nabla u_\varep\|_{L^p(\Omega)}
\le C\, \| g\|_{B^{-1/p, p}(\partial\Omega)}$
follows from (\ref{4.2.1})-(\ref{4.2.2}) by duality. 
\end{proof}

Let $1<q<d$ and $\frac{1}{p}=\frac{1}{q}-\frac{1}{d}$.
In the proof of the next lemma, we will need the following Sobolev inequality 
\begin{equation}\label{Sobolev}
\left(\int_\Omega |u|^p\, dx \right)^{1/p}
\le C\left(\int_\Omega |\nabla u|^q \, dx \right)^{1/q},
\end{equation}
where $u\in W^{1,q}(\Omega)$ and $\int_{\partial\Omega} u =0$.

\begin{lemma}\label{lemma-4.4}
Suppose that $A\in \Lambda(\mu, \lambda,\tau)$.
Let $F\in L^q(\Omega)$, where $1<q<d$ and $\Omega$ is a bounded $C^{1, \alpha_0}$
domain.
Let $u\in W^{1,p}(\Omega)$ be a weak solution to
$\mathcal{L}_\varep (u_\varep)=F$ in $\Omega$ and $\frac{\partial u_\varep}{\partial\nu_\varep}
=-b$ on $\partial\Omega$, where $\frac{1}{p}=\frac{1}{q}-\frac{1}{d}$
and $b=\frac{1}{|\partial\Omega|}\int_\Omega F$.
Then $\|\nabla u_\varep\|_{L^p(\Omega)} \le C\, \| F\|_{L^q(\Omega)}$.
\end{lemma}

\begin{proof}
Let $f\in C_0^\infty(\Omega)$ and $v_\varep$ be a weak solution to
$(\mathcal{L}_\varep)^* (v_\varep)=\text{\rm div} (f)$ in $\Omega$
and $\frac{\partial v_\varep}{\partial\nu_\varep}=0$ on $\partial\Omega$.
By Lemma \ref{lemma-4.1}, we have $\|\nabla v_\varep\|_{L^{p^\prime}(\Omega)}
\le C\, \| f\|_{L^{p^\prime}(\Omega)}$.
Note that
\begin{equation}\label{4.4.1}
\aligned
\int_\Omega \frac{\partial u_\varep^\alpha}{\partial x_i} \cdot f_i^\alpha\, dx
& =\int_\Omega a_{ij}^{\alpha\beta}\left(\frac{x}{\varep}\right) \frac{\partial u_\varep^\beta}
{\partial x_j}\cdot \frac{\partial v_\varep^\alpha}{\partial x_i}\, dx\\
&=\int_\Omega F\cdot v_\varep\, dx
-\int_{\partial\Omega} b\cdot v_\varep\, d\sigma\\
&=\int_\Omega F (v_\varep -E)\, dx,
\endaligned
\end{equation}
where $E$ is the average of $v_\varep$ over $\partial\Omega$.
It follows from (\ref{4.4.1}) and Sobolev inequality (\ref{Sobolev}) that
$$
\aligned
\big|\int_\Omega \frac{\partial u_\varep^\alpha}{\partial x_i} \cdot f_i^\alpha\, dx\big|
& \le \|F\|_{L^q(\Omega)}
\| v_\varep -E\|_{L^{q^\prime}(\Omega)}\\
& \le C \| F\|_{L^q(\Omega)} \|\nabla v_\varep\|_{L^{p^\prime}(\Omega)}\\
&\le C \| F\|_{L^q(\Omega)}\| f\|_{L^{p^\prime}(\Omega)}.
\endaligned
$$
By duality this gives $\|\nabla u_\varep\|_{L^p(\Omega)} \le C\, \|F\|_{L^q(\Omega)}$.
\end{proof}

\noindent{\bf Proof of Theorem \ref{W-1-p-theorem}.}
Let $v_\varep$ be a weak solution to 
$\mathcal{L}_\varep (v_\varep) =\text{\rm div} (f)$ in $\Omega$
and $\frac{\partial v_\varep}{\partial \nu_\varep} = -n\cdot f$ on
$\partial\Omega$.
Let $w_\varep$ be a weak solution to
$\mathcal{L}_\varep (w_\varep) =F$ in $\Omega$ and
$\frac{\partial w_\varep}{\partial \nu_\varep} = -b$ on
$\partial\Omega$, where $b=\frac{1}{|\partial\Omega|}\int_\Omega F$.
Finally, let $h_\varep =u_\varep -v_\varep-w_\varep$. Then
$\mathcal{L}_\varep (h_\varep) =0$ in $\Omega$ and
$\frac{\partial h_\varep}{\partial\nu_\varep}
=g+b$ on $\partial\Omega$.
It follows from Lemmas \ref{lemma-4.1}, \ref{lemma-4.2} and \ref{lemma-4.4} that
$$
\aligned
\|\nabla u_\varep\|_{L^p(\Omega)}
& \le \|\nabla v_\varep\|_{L^p(\Omega)}
+\| \nabla w_\varep\|_{L^p(\Omega)} +\|\nabla h_\varep\|_{L^p(\Omega)}\\
& \le C\, 
\left\{ \| f\|_{L^p(\Omega)}
+\| F\|_{L^q(\Omega)}
+\| g\|_{B^{-1/p, p}(\partial\Omega)}\right\},
\endaligned
$$
where $q=\frac{pd}{p+d}$ for $p>\frac{d}{d-1}$, and $q>1$ for $1<p\le \frac{d}{d-1}$.
This completes the proof.
\qed

\section{A matrix of Neumann functions}

Let $\Gamma_\varep (x,y)
=\big(\Gamma_{A,\varep}^{\alpha\beta} (x,y)\big)_{m\times m}$
denote the matrix of fundamental solutions of $\mathcal{L}_\varep$ in $\br^d$,
with pole at $y$. Under the assumption $A\in \Lambda(\mu,\lambda, \tau)$,
one may use the interior estimate (\ref{interior-estimate}) to show that
for $d\ge 3$,
\begin{equation}\label{fundamental-estimate-1}
|\Gamma_\varep (x,y)|\le C|x-y|^{2-d}
\end{equation}
and
\begin{equation}\label{fundamental-estimate-2}
|\nabla_x\Gamma_\varep (x,y)| +|\nabla_y \Gamma_\varep (x,y)|\le C |x-y|^{1-d},
\end{equation}
where
$C$ depends only on $d$, $m$, $\mu$, $\lambda$ and $\tau$
(see e.g. \cite{Hofmann-Kim-2007}; the size estimate (\ref{fundamental-estimate-1})
also follows from \cite{ERS}).
Let $V_\varep (x,y) =\big(V_{A,\varep}^{\alpha\beta}(x,y)\big)_{m\times m}$, 
where for each $y\in \Omega$,
$V_\varep^\beta
(x,y)=\big(V_{A,\varep}^{1\beta}(x,y),\dots, V_{A,\varep}^{m\beta}(x,y)\big)$ solves
\begin{equation}\label{definition-of-V}
\left\{
\aligned
\mathcal{L}_\varep \big( V_\varep^\beta (\cdot, y)\big) & =0  \qquad\text{ in } \Omega,\\
\frac{\partial}{\partial \nu_\varep}
\big\{ V_\varep^\beta (\cdot, y)\big\} & =\frac{\partial}{\partial\nu_\varep}
\big\{ \Gamma_\varep^\beta (\cdot, y)\big\} +\frac{e^\beta}{|\partial\Omega|}  \qquad
\text{ on }\partial\Omega,\\
\int_{\partial\Omega} V_\varep^{\beta} (x,y)\, d\sigma (x)
& =\int_{\partial\Omega}\Gamma_\varep^{\beta} (x,y)\, d\sigma (x),
\endaligned
\right.
\end{equation}
where $\Gamma_\varep^\beta (x,y)=(\Gamma_{A,\varep}^{1\beta}(x,y),
 \dots, \Gamma_{A,\varep}^{m\beta}
(x,y))$
and $e^\beta =(0, \dots, 1, \dots, 0)$ with $1$ in the $\beta^{th}$ position.
We now define
\begin{equation}\label{definition-of-N}
N_\varep (x,y)=\big( N^{\alpha\beta}_{A,\varep} (x,y)\big)_{m\times m}
=\Gamma_\varep (x,y) -V_\varep (x,y),
\end{equation}
for $x,y\in \Omega$.
Note that, if $N_\varep^\beta (x,y) =\Gamma_\varep^\beta (x,y)-V_\varep^\beta (x,y)$,
\begin{equation}\label{equation-for-N}
\left\{
\aligned
\mathcal{L}_\varep \big\{ N^\beta_\varep (\cdot, y)\} & 
=e^\beta\delta_y(x) \qquad\text{ in } \Omega,\\
\frac{\partial}{\partial\nu_\varep} \big\{ N^\beta_\varep (\cdot, y)\big\}
& =-e^\beta |\partial\Omega|^{-1} \qquad \text{ on } \partial\Omega\\
\int_{\partial\Omega} N^\beta_\varep (x,y)\, d\sigma (x) & =0,
\endaligned
\right.
\end{equation}
where $\delta_y (x)$ denotes the Dirac delta function with pole at $y$.
We will call $N_\varep (x,y)$ the matrix of Neumann functions for
$\mathcal{L}_\varep$ in $\Omega$.

\begin{lemma}\label{symmetry-lemma}
For any $x,y\in \Omega$, we have
\begin{equation}\label{Neumann-symmetry}
N_{A,\varep}^{\alpha\beta} (x,y)
=N_{A^*, \varep}^{\beta\alpha}(y,x),
\end{equation}
where $A^*$ denotes the adjoint of $A$.
\end{lemma}

\begin{proof}
Note that
\begin{equation}\label{fundamental-symmetry}
\Gamma_{A,\varep}^{\alpha\beta} (x,y)
=\Gamma_{A^*, \varep}^{\beta\alpha} (y,x),
\qquad \text{ for any } x,y\in \Omega.
\end{equation}
Using the Green's representation formula for $\mathcal{L}_\varep$ on $\Omega$,
(\ref{definition-of-V}) and (\ref{equation-for-N}) one may show that
$$
\aligned
&
V_{A,\varep}^{\alpha\beta} (x,y)
+\Gamma_{A, \varep}^{\alpha\beta}(x,y)
-\frac{1}{|\partial\Omega|}
\int_{\partial\Omega}
\left\{ \Gamma_{A,\varep}^{\alpha\beta} (z,y)
+\Gamma_{A^*, \varep}^{\beta\alpha} (z,x)\right\}\, d\sigma (z)\\
&
=\int_\Omega a_{ij}^{\gamma\delta} \left(\frac{z}{\varep}\right)
\frac{\partial}{\partial z_i}
\bigg\{ \Gamma_{A^*, \varep}^{\gamma\alpha}(z,x)\bigg\}
\cdot \frac{\partial}{\partial z_j}
\bigg\{ \Gamma_{A, \varep}^{\delta\beta}(z,y)\bigg\}\, dz\\
&\qquad
-
\int_\Omega a_{ij}^{\gamma\delta} \left(\frac{z}{\varep}\right)
\frac{\partial}{\partial z_i}
\bigg\{ V_{A^*, \varep}^{\gamma\alpha}(z,x)\bigg\}
\cdot \frac{\partial}{\partial z_j}
\bigg\{ V_{A, \varep}^{\delta\beta}(z,y)\bigg\}\, dz.
\endaligned
$$
This gives $V_{A,\varep}^{\alpha\beta}(x,y)
=V_{A^*, \varep}^{\beta\alpha}(y,x)$ and hence (\ref{Neumann-symmetry}).
\end{proof}

\begin{thm}\label{Neumann-theorem-5.2}
Let $\Omega$ be a bounded $C^{1,\alpha_0}$ domain and
 $A\in \Lambda(\mu, \lambda, \tau)$.
Let $x_0,y_0, z_0\in \Omega$ be such that $|x_0-z_0|<(1/4)|x_0-y_0|$.
Then for any $\gamma\in (0,1)$,
\begin{equation}\label{5.2}
\left\{
\average_{B(y_0,\rho/4)\cap\Omega}
\big|\nabla_y\big\{ N_\varep (x_0,y) -N_\varep (z_0,y)\big\} |^2\, dy\right\}^{1/2}
\le C \rho^{1-d} \left(\frac{|x_0-z_0|}{\rho}\right)^{\gamma},
\end{equation}
where $\rho=|x_0-y_0|$ and 
$C$ depends only on $\mu$, $\lambda$, $\tau$, $\gamma$ and $\Omega$. 
\end{thm}

\begin{proof}
Let $f\in C_0^\infty (B(y_0,\rho/2)\cap\Omega)$ and $\int_\Omega f=0$.
Let
$$
u_\varep (x)
=\int_\Omega N_\varep (x,y) f(y)\, dy.
$$
Then $\mathcal{L}_\varep (u_\varep)=f$ in $\Omega$ and $\frac{\partial u_\varep}{\partial\nu_\varep}
=0$ on $\partial\Omega$. Since $\mathcal{L}_\varep (u_\varep)
=0$ in $B(x_0,\rho/2)\cap\Omega$, it follows 
from the boundary H\"older estimate (\ref{boundary-holder-estimate})
and interior estimates
that
\begin{equation}\label{5.2.1}
|u_\varep (x_0)-u_\varep (z_0)|
\le C
\left(\frac{|x_0-z_0|}{\rho}\right)^{\gamma}\cdot\rho
\cdot \left\{
\average_{B(x_0, \rho/2)\cap\Omega}
|\nabla u_\varep |^2\right\}^{1/2}.
\end{equation}
Let $E$ be the average of $u_\varep$ over $B(y_0,\rho/2)\cap\Omega$.
Note that by (\ref{ellipticity}),
\begin{equation}\label{5.2.2}
\aligned
\mu \int_\Omega |\nabla u_\varep|^2 dx \le
& \left| \int_\Omega f\cdot u_\varep\, dx \right|
=\left|\int_{B(y_0, \rho/2)\cap\Omega}
f\cdot (u_\varep -E)\, dx\right|\\
& \le \| f\|_{L^2(\Omega)}
\| u_\varep -E\|_{L^2(B(y_0,\rho/2)\cap\Omega)}\\
& \le C \rho  \| f\|_{L^2(\Omega)}
\| \nabla u_\varep\|_{L^2(B(y_0,\rho/2)\cap\Omega)},
\endaligned
\end{equation}
where we have used the Cauchy and Poincar\'e inequalities. 
Hence, $\|\nabla u_\varep\|_{L^2(\Omega)}
\le C\rho \|f\|_{L^2(\Omega)}$.
This, together with (\ref{5.2.1}), gives
$$
|u_\varep (x_0)-u_\varep (z_0)|
\le C\rho^{2-\frac{d}{2}} \left(\frac{|x_0-z_0|}{\rho}\right)^\gamma
\| f\|_{L^2(\Omega)}.
$$
By duality this implies that
\begin{equation}\label{5.2.3}
\left\{ \int_{B(y_0,\rho/2)\cap\Omega}
\big| W(y) -C_{x_0,z_0}\big|^2\, dy\right\}^{1/2}
\le C\rho^{2-\frac{d}{2}} 
\left(\frac{|x_0-z_0|}{\rho}\right)^\gamma,
\end{equation}
where $W(y)=N_\varep (x_0,y)-N_\varep (z_0,y)$ and $C_{x_0,z_0}$ 
is the average of $W$ 
over $B(y_0,\rho/2)\cap\Omega$.
In view of (\ref{Neumann-symmetry}) we have
 $(\mathcal{L}_\varep)^* (W^*)=0$ in $B(y_0, \rho/2)\cap\Omega$
and $\frac{\partial }{\partial \nu_\varep^*}\{ W^*\} =0$ on $\partial\Omega$, where
$\frac{\partial}{\partial\nu^*_\varep}$ denote the conormal derivative
associated with $(\mathcal{L}_\varep)^*$.
The estimate (\ref{5.2}) now follows from (\ref{5.2.3}) by
Cacciopoli's inequality (\ref{Cacciopoli}).
\end{proof}

\begin{lemma}\label{lemma-5.3}
Let $V_\varep(x,y)$ be defined by (\ref{definition-of-V}).
Suppose $d\ge 3$. Then for any $x,y\in \Omega$,
\begin{equation}\label{estimate-5.3}
|V_\varep (x,y)|\le {C}{\big[ \delta (x)\big]^{\frac{2-d}{2}}
\big[\delta(y)\big]^{\frac{2-d}{2}}},
\end{equation}
where $\delta (x)=\text{\rm dist}(x,\partial\Omega)$.
\end{lemma}

\begin{proof}
We begin by fixing $y\in\Omega$ and $1\le \beta\le m$. 
Let $u_\varep (x)=V_\varep (x,y)$.
In view of (\ref{definition-of-V}) we have
$$
\|\nabla u_\varep \|_{L^2(\Omega)}
\le C \|\frac{\partial u_\varep}{\partial\nu_\varep}\|_{W^{-1/2,2} (\partial\Omega)}
\le C \|\frac{\partial u_\varep}{\partial\nu_\varep}\|_{L^p(\partial\Omega)},
$$
where $p=\frac{2(d-1)}{d}$. Note that by (\ref{fundamental-estimate-2}),
$$
\aligned
\|\frac{\partial u_\varep}{\partial\nu_\varep}\|_{L^p(\partial\Omega)}
 & \le C \left\{ \int_{\partial\Omega}
\frac{d\sigma (x)}{|x-y|^{p(d-1)}}\right\}^{1/p}
+ C|\partial\Omega|^{\frac{1}{p}-1}\\
& \le C\big[\delta (y)\big]^{\frac{2-d}{2}}.
\endaligned
$$
Thus we have proved that
$$
\| \nabla u_\varep\|_{L^2(\Omega)}
\le C \big[ \delta (y)\big]^{\frac{2-d}{2}}.
$$

Now, by the interior estimates and the Sobolev inequality (\ref{Sobolev}),
$$
\aligned
|u_\varep (x)|
&\le C\left\{ \frac{1}{[\delta (x)]^d}
\int_{B(x,\delta(x)/2)}
|u_\varep (z)|^{2^*} dz \right\}^{1/2^*}\\
&\le C \big[ \delta (x)\big]^{\frac{2-d}{2}}
\left\{ \left(\int_\Omega |\nabla u_\varep|^2\, dx \right)^{1/2}
+|\Omega|^{\frac{1}{2^*}}
\big| \average_{\partial\Omega} u_\varep d\sigma\big| \right\}\\
&\le
C \big[ \delta (x)\big]^{\frac{2-d}{2}}
\left\{ \big[ \delta (y)\big]^{\frac{2-d}{2}}
+|\Omega|^{\frac{1}{2^*}}
\big| \average_{\partial\Omega} \Gamma_\varep (z,y) d\sigma (z)\big| \right\}\\
&\le
C \big[ \delta (x)\big]^{\frac{2-d}{2}}
\big[ \delta (y)\big]^{\frac{2-d}{2}},
\endaligned
$$
where $2^*=\frac{2d}{d-2}$.
\end{proof}

\begin{thm}\label{Neumann-function-theorem}
Let $\Omega$ be a bounded $C^{1,\alpha}$ domain in $\brd$, $d\ge 3$.
Suppose that $A\in \Lambda (\mu, \lambda, \tau)$.
Then
\begin{equation}\label{Neumann-size-estimate}
|N_\varep (x,y)|\le C |x-y|^{2-d}
\end{equation}
and for any $\gamma\in (0,1)$,
\begin{equation}\label{Neumann-holder-estimate}
\aligned
|N_\varep (x,y)-N_\varep (z,y)| & \le \frac{C_\gamma |x-z|^\gamma}{|x-y|^{d-2+\gamma}},\\
|N_\varep (y,x)-N_\varep (y,z)| & \le \frac{C_\gamma |x-z|^\gamma}{|x-y|^{d-2+\gamma}},
\endaligned
\end{equation}
where $|x-z|<(1/4)|x-y|$.
\end{thm}

\begin{proof}
By Theorem \ref{boundary-holder-theorem} 
we only need to establish the size estimate (\ref{Neumann-size-estimate}).
To this end we first note that by Lemma \ref{lemma-5.3},
\begin{equation}\label{5.4.1}
|N_\varep (x,y)|\le C\big\{ |x-y|^{2-d}
+\big[\delta(x)\big]^{2-d}
+\big[\delta(y)\big]^{2-d}\big\}.
\end{equation}
Next, let $\rho=|x-y|$. It follows from Theorem \ref{boundary-holder-theorem}
 and (\ref{5.4.1}) that
\begin{equation}\label{5.4.2}
\aligned
|N_\varep (x,y)|
 & \le C \left\{\left\{
\average_{B(x,\rho/4)\cap\Omega} |N_\varep (z,y)|^p \, dz\right\}^{1/p}
+\frac{\rho}{|\partial\Omega|}\right\}\\
& \le C
\big\{
|x-y|^{2-d}
+\big[\delta (y)\big]^{2-d}\big\},
\endaligned
\end{equation}
where we have chosen $p$ so that $p(d-2)<1$.
With estimate (\ref{5.4.2}) at our disposal,
another application of Theorem \ref{boundary-holder-theorem} gives
$$
\aligned
|N_\varep (x,y)|
&\le C\left\{ \left\{\average_{B(y,\rho/4)\cap \Omega}
|N_\varep (x,z)|^p\, dz\right\}^{1/p}
+\rho^{2-d}\right\}\\
&\le C|x-y|^{2-d}.
\endaligned
$$
This finishes the proof.
\end{proof}

\begin{remark}\label{remark-5.0}
{\rm 
If $m=1$ and $d\ge 3$, the size estimate (\ref{Neumann-size-estimate}) and H\"older estimate
(\ref{Neumann-holder-estimate})
for some $\gamma>0$ were established in \cite{Kenig-Pipher-1993}
for divergence form elliptic operators with bounded measurable coefficients in 
bounded star-like Lipschitz domains.
}
\end{remark}

\begin{remark}\label{remark-5.2} 
{\rm 
Suppose that $d\ge 3$.
The matrix of Neumann functions for the exterior domain $\Omega_-=\br^d\setminus \overline{\Omega}$
may be constructed in a similar fashion. Indeed, let $N_\varep^- (x,y)
=\Gamma_\varep (x,y)-V_\varep^-(x,y)$, where $V_\varep^-(x,y)$ is chosen so that
for each $y\in \Omega_-$,
\begin{equation}\label{5.6.1}
\left\{
\aligned
\mathcal{L}_\varep \big\{ N_\varep^- (\cdot, y)\big\} & =\delta_y (x) I \quad \text{ in }\Omega,\\
\frac{\partial}{\partial \nu_\varep} \big\{ N_\varep^- (\cdot, y)\big\}
&=0 \quad \text{ on }\partial\Omega,\\
N_\varep^- (x,y) & =O(|x-y|^{2-d}) \quad \text{ as } |x|\to\infty,
\endaligned
\right.
\end{equation}
where $I$ is the $m\times m$ identity matrix.
The estimates in Theorem \ref{Neumann-function-theorem} continue to hold
for $N_\varep^-(x,y)$.
}
\end{remark}

\begin{remark}\label{remark-5.1}
{\rm 
If $d=2$, the matrix of Neumann functions may be defined as follows.
Choose $B(0,R)$ such that $\Omega\subset B(0,R/2)$. Let $G_\varep (x,y)$ be
the Green's function for $\mathcal{L}_\varep$ in $ B(0,R)$.
Define $N_\varep (x,y)=G_\varep (x,y)-V_\varep (x,y)$, where $V_\varep (x,y)$ is the solution
to (\ref{definition-of-V}), but with $\Gamma_\varep (x,y)$ replaced by $G_\varep (x,y)$.
Theorem \ref{Neumann-theorem-5.2} continues to hold for $d=2$.
One may modify the argument in the proof of Lemma \ref{lemma-5.3}
to show that
$$
|V_\varep (x,y)|\le C_\gamma \big[\delta(x)\big]^{-\gamma} \big[\delta(y)\big]^{-\gamma},
$$
for any $\gamma>0$. In view of the proof of Theorem \ref{Neumann-function-theorem}
and the estimate 
$|G_\varep (x,y)|\le C \big\{ 1+\big|\ln |x-y|\big|\big\}$ in \cite{AL-1987}, 
this gives
$
|N_\varep (x,y)|\le C_\gamma|x-y|^{-\gamma}$
for any $\gamma>0$. 
}
\end{remark}

\section{Correctors for Neumann boundary conditions}

Let $\Phi_\varep =(\Phi_{\varep,j}^{\alpha\beta})$, 
where for each $1\le j\le d$ and $1\le \beta\le m$,
$\Phi_{\varep,j}^\beta =(\Phi_{\varep,j}^{1\beta}, \dots, \Phi_{\varep, j}^{m \beta})$ 
 is a solution to the Neumann problem
\begin{equation}\label{Phi}
\left\{
\aligned
\mathcal{L}_\varep \big( \Phi_{\varep,j}^\beta) & =0 &\quad & \text{ in } \Omega,\\
\frac{\partial}{\partial\nu_\varep}
\big(\Phi_{\varep,j}^\beta\big) & =\frac{\partial}{\partial \nu_0} \big( P_j^\beta\big) & \quad
&\text{ on } \partial \Omega,\\
\endaligned
\right.
\end{equation}
Here $P_j^\beta =P_j^\beta (x)=x_j (0,\dots, 1, \dots, 0)$ with $1$ in the $\beta^{th}$
position.
In the study of boundary estimates for Neumann boundary conditions,
 the function $\Phi_\varep (x)-x$
plays a similar role as $\varep \chi (\frac{x}{\varep})$ for interior 
estimates. The goal of this section is to prove the following uniform Lipschitz estimate of
$\Phi_\varep$.

\begin{thm}\label{corrector-theorem}
Let $\Omega$ be a $C^{1,\alpha_0}$ domain.
Suppose that $A\in \Lambda (\mu,\lambda, \tau)$ and $A^*=A$.
Then
\begin{equation}\label{corrector-estimate}
\|\nabla \Phi_\varep\|_{L^\infty(\Omega)} \le C,
\end{equation}
where $C$ depends only on $d$, $m$, $\mu$, $\lambda$, $\tau$ and $\Omega$.
\end{thm}

Our proof of Theorem \ref{corrector-theorem} uses the uniform $L^2$ Rellich
estimate for Neumann problem:
\begin{equation}\label{Rellich-estimate}
\int_{\partial\Omega} |\nabla u_\varep|^2\, d\sigma
\le C\int_{\partial\Omega}
\big|\frac{\partial u_\varep}{\partial \nu_\varep}\big|^2\, d\sigma,
\end{equation}
for solutions of $\mathcal{L}_\varep (u_\varep)=0$ in $\Omega$.
We mention that  (\ref{Rellich-estimate}) as well as the uniform $L^2$ Rellich estimate
for the regularity of Dirichlet problem:
\begin{equation}\label{Rellich-estimate-1}
\int_{\partial\Omega} |\nabla u_\varep|^2\, d\sigma
\le C\int_{\partial\Omega}
\big|\nabla_{\rm tan} u_\varep\big|^2\, d\sigma,
\end{equation}
was established by Kenig and Shen
in \cite{Kenig-Shen-2} under the assumption that $\Omega$ is Lipschitz,
$A\in \Lambda (\mu, \lambda,\tau)$
and $A=A^*$ (also see \cite{Kenig-Shen-1} for the case of the elliptic equation).
The constant $C$ in (\ref{Rellich-estimate})-(\ref{Rellich-estimate-1})
depends only on $d$, $m$,
$\mu$, $\lambda$, $\tau$ and the Lipschitz character of $\Omega$.

\begin{lemma}\label{6.2-lemma}
Let $\Omega$ and $\mathcal{L}$ satisfy the same assumptions as in Theorem \ref{corrector-theorem}.
Suppose that $\mathcal{L}_\varep (u_\varep)=0$ in $\Omega$,
$\frac{\partial u_\varep}{\partial\nu_\varep}=g$ on $\partial\Omega$, and
$$
g=\sum_{i,j}
\left( n_i \frac{\partial}{\partial x_j}-
n_j\frac{\partial}{\partial x_i}\right) g_{ij},
$$
where $g_{ij}\in C^1(\partial\Omega)$ and
$n=(n_1, \dots, n_d)$ denotes the unit outward normal to $\partial\Omega$.
Then
\begin{equation}\label{6.2.1}
|\nabla u_\varep (x)|
\le \frac{C}{\delta (x)}
\sum_{i,j} \| g_{ij}\|_{L^\infty (\partial \Omega)},
\end{equation}
for any $x\in \Omega$, where $\delta(x)=\text{\rm dist}(x,\partial\Omega)$.
\end{lemma}

\begin{proof}
By the interior estimate (\ref{interior-estimate}) we only need to show that
\begin{equation}\label{6.2.2-1}
|u_\varep (x)-u_\varep (z)|
\le C \sum_{i,j} \| g_{ij}\|_{L^\infty (\partial \Omega)},
\end{equation}
where $|x-z|\le cr$ and $r=\delta (x)$.
Let $N_\varep (x,y)$ denote the matrix of Neumann functions for $\mathcal{L}_\varep$
on $\Omega$.
Note that
$$
\aligned
u_\varep (x)-u_\varep (z)
&=\int_{\partial\Omega}
\big\{ N_\varep (x,y)-N_\varep (z,y)\big\} g(y)\, d\sigma(y)\\
&=
\int_{\partial\Omega}
\big\{ N_\varep (x,y)-N_\varep (z,y)\big\} 
\sum_{i,j}
\left( n_i \frac{\partial}{\partial y_j}-
n_j\frac{\partial}{\partial y_i}\right) g_{ij} (y)
\, d\sigma (y)\\
& =-\sum_{i,j}
\int_{\partial\Omega}
\left( n_i \frac{\partial}{\partial y_j}-
n_j\frac{\partial}{\partial y_i}\right)
\big\{ N_\varep (x,y)-N_\varep (z,y)\big\} 
\cdot g_{ij}(y)\, d\sigma (y),
\endaligned
$$
where we have used the fact that $n_i\frac{\partial}{\partial y_j}
-n_j\frac{\partial}{\partial y_i}$ is a tangential derivative
on $\partial\Omega$.
Consequently it suffices to show that
\begin{equation}\label{6.2.2}
\int_{\partial\Omega}
\big| \nabla _y \big\{ N_\varep (x,y)-N_\varep (z,y)\big\} \big|\,
d\sigma (y) \le C,
\end{equation}
if $|x-z|\le cr$ and $r=\delta(x)$.

Let $Q\in \partial\Omega$ so that $|x-Q|=\text{\rm dist}(x, \partial\Omega)$.
By translation and rotation we may assume that $Q=0$ and
$$
\aligned
&\Omega\cap \{ (x^\prime, x_d): \ |x^\prime|<8cr \text{ and } |x_d|<8cr\}\\
&\quad=\big\{ (x^\prime, x_d): \ |x^\prime|<8cr \text{ and } 
\psi(x^\prime)<x_d <8cr\}
\endaligned
$$
where $\psi(0)=|\nabla \psi (0)|=0$ and $c$ is sufficiently small.
To establish (\ref{6.2.2}) we will show that 
\begin{equation}\label{6.2.3}
\int_{|y|\le cr}
\big| \nabla _y \big\{ N_\varep (x,y)-N_\varep (z,y)\big\} \big|\,
d\sigma (y) \le C,
\end{equation}
and there exists $\beta>0$ such that for $cr<\rho<r_0 $,
\begin{equation}\label{6.2.4}
\int_{ |y-P|\le c\rho}
\big| \nabla _y \big\{ N_\varep (x,y)-N_\varep (z,y)\big\} \big|\,
d\sigma (y) \le C\left(\frac{r}{\rho}\right)^\beta,
\end{equation}
where $P\in\partial\Omega$ and $|P|=\rho$.
The estimate (\ref{6.2.2})
follows from (\ref{6.2.3}) and (\ref{6.2.4})
by a simple covering argument.

To see (\ref{6.2.3}) we let 
$$
S(t)=\big\{ (x^\prime, x_d): \ |x^\prime|<t \text{ and } \psi(x^\prime)
<x_d< \psi(x^\prime) +ct\big\}.
$$
Note that by Cauchy inequality, for $t\in (cr,2cr)$,
\begin{equation}\label{6.2.5}
\aligned
& \left\{ \int_{|y|\le cr}
\big| \nabla _y \big\{ N_\varep (x,y)-N_\varep (z,y)\big\} \big|\,
d\sigma (y)\right\}^2\\
&\qquad
\le Cr^{d-1}
 \int_{\partial S(t)}
\left|
\nabla_y
\big\{ N_\varep (x,y)-N_\varep (z,y)\big\} \right|^2\, d\sigma (y)\\
&\qquad
\le Cr^{d-1}
\int_{\partial S(t)}
\left|
\frac{\partial }{\partial \nu^*_\varep}
\big\{ N_\varep (x,y)-N_\varep (z,y)\big\} \right|^2\, d\sigma (y),
\endaligned
\end{equation}
where  we have used the Rellich estimate (\ref{Rellich-estimate})
for the last inequality.
Since
$$
\frac{\partial }{\partial \nu^*_\varep (y)}
\big\{ N_\varep (x,y)-N_\varep (z,y)\big\} =0 \quad \text{ in }\partial\Omega,
$$
we may integrate both sides of (\ref{6.2.5}) in $t$ over $(cr, 2cr)$ to obtain
\begin{equation}\label{6.2.6}
\aligned
& \left\{ \int_{|y|\le cr}
\big| \nabla _y \big\{ N_\varep (x,y)-N_\varep (z,y)\big\} \big|\,
d\sigma (y)\right\}^2 \\
&\qquad
\le Cr^{d-2}
 \int_{ S(2cr)}
\left|
\nabla_y
\big\{ N_\varep (x,y)-N_\varep (z,y)\big\} \right|^2\, dy.
\endaligned
\end{equation}
The desired estimate (\ref{6.2.3}) now follows from estimate
(\ref{5.2}).

The proof of (\ref{6.2.4}) is similar to that of (\ref{6.2.3}).
Indeed, an analogous argument gives
$$
\aligned
& \left\{ \int_{|y-P|\le c\rho}
\left|
\nabla_y
\big\{ N_\varep (x,y)-N_\varep (z,y)\big\} \right|^2\, d\sigma(y)\right\}^2 \\
&\qquad
\le C\rho^{d-2}
\int_{|y-P|\le 2c\rho}
\left|
\nabla_y
\big\{ N_\varep (x,y)-N_\varep (z,y)\big\} \right|^2\, dy\\
&\qquad
\le C \left(\frac{r}{\rho}\right)^{2\gamma}.
\endaligned
$$
This completes the proof.
\end{proof}

Let $\Psi_\varep =\big(\Psi_{\varep,j}^{\alpha\beta} (x)\big)$, 
where $1\le j\le d$, $1\le\alpha, \beta\le m$ and
\begin{equation}\label{Psi}
\Psi_{\varep,j}^{\alpha\beta}(x)
=\Phi_{\varep,j}^{\alpha\beta} (x)
-x_j \delta_{\alpha\beta} -\varep \chi_j^{\alpha\beta} \left(\frac{x}{\varep}\right).
\end{equation}

\begin{lemma}\label{6.3-lemma}
Suppose that
 $\Omega$ and $\mathcal{L}$ satisfy the same conditions as in Theorem \ref{corrector-theorem}.
Then
\begin{equation}\label{6.3.1}
|\nabla \Psi_\varep (x)|\le \frac{C\varep}{\delta(x)}
\qquad \text{ for any } x\in\Omega.
\end{equation}
\end{lemma}

\begin{proof}
Fix $1\le \ell \le d$ and $1\le \gamma \le m$.
Let $w=(w^1, \dots, w^m)=(\Psi_{\varep,\ell}^{1\gamma}, \dots, \Psi_{\varep,\ell}^{m\gamma})$.
Note that $\mathcal{L}_\varep (w)=0$ in $\Omega$.
In view of Lemma \ref{6.2-lemma} it suffices to show that
there exists $ g_{ij}\in C^1(\partial\Omega)$ such that 
\begin{equation}\label{6.3.2}
\left\{ 
\aligned
& \frac{\partial w}{\partial \nu_{\varep}}
=\sum_{i,j}
\left(n_i \frac{\partial}{\partial x_j}
-n_j \frac{\partial}{\partial x_i}\right) g_{ij},\\
& \| g_{ij}\|_{L^\infty(\partial\Omega)}\le C\varep.
\endaligned
\right.
\end{equation}

To this end we observe that by the definition of $\Phi_{\varep,j}^{\alpha\beta}$ in (\ref{Phi}),
$$
\aligned
\left( \frac{\partial w}{\partial\nu_\varep}\right)^\alpha
& =n_i a_{ij}^{\alpha\beta} \left(\frac{x}{\varep}\right)
\frac{\partial }{\partial x_j}
\left\{ \Phi_{\varep,\ell}^{\beta\gamma}\right\}
-n_i a_{ij}^{\alpha\beta} \left(\frac{x}{\varep}\right)
\frac{\partial}{\partial x_j}
\left\{ x_\ell \delta_{\beta\gamma}
+\varep \chi_\ell^{\beta\gamma} \left(\frac{x}{\varep}\right) \right\}\\
&
=n_i \hat{a}_{ij}^{\alpha\beta}
\frac{\partial}{\partial x_j}
\big\{ x_\ell \delta_{\beta\gamma}\big\}
-
n_i a_{ij}^{\alpha\beta} \left(\frac{x}{\varep}\right)
\frac{\partial}{\partial x_j}
\left\{ x_\ell \delta_{\beta\gamma}
+\varep \chi_\ell^{\beta\gamma} \left(\frac{x}{\varep}\right)\right\}\\
&=
n_i \hat{a}^{\alpha\gamma}_{i\ell}
-n_i a_{ij}^{\alpha\beta}
\left(\frac{x}{\varep}\right)
\left\{ \delta_{j\ell}\delta_{\beta\gamma}
+\frac{\partial \chi_\ell^{\beta\gamma}}{\partial x_j} \left(\frac{x}{\varep}\right)\right\},
\endaligned
$$
where $\hat{a}_{ij}^{\alpha\beta}$ are the homogenized coefficients defined 
by (\ref{homogenized-coefficient}).
Let
\begin{equation}\label{6.3.3}
H_{i\ell}^{\alpha\gamma} (y)
=\hat{a}^{\alpha\gamma}_{i\ell}
-a_{ij}^{\alpha\beta} (y)
\left\{ 
\delta_{j\ell}\delta_{\beta\gamma}
+\frac{\partial \chi_\ell^{\beta\gamma}}{\partial y_j} (y)\right\}.
\end{equation}
It follows from the definition of $\hat{a}_{i\ell}^{\alpha\gamma}$ that
$$
\int_{[0,1]^d} H_{i\ell}^{\alpha\gamma} (y)\, dy =0.
$$
Thus we may solve the Poisson equation on $[0,1]^d$ with periodic boundary
conditions,
\begin{equation}\label{6.3.4}
\left\{
\aligned
& \Delta U_{i\ell}^{\alpha\gamma} =H_{i\ell}^{\alpha\gamma} \quad \text{ in } \brd,\\
& U_{i\ell}^{\alpha\gamma} (y) \text{ is periodic with respect to } \mathbb{Z}^d.
\endaligned
\right.
\end{equation}
Since $A(y)$ and $\nabla\chi (y)$ are H\"older continuous,
$\nabla^2 U_{i\ell}^{\alpha\gamma}$ is H\"older continuous.
In particular, we have $\|\nabla U_{i\ell}^{\alpha\gamma}\|_\infty
\le C$, where $C$ depends only on $\mu$, $\lambda$ and $\tau$.

Now let
$$
F_{i\ell k}^{\alpha\gamma} (y)
=\frac{\partial }{\partial y_k} \bigg\{ U_{i\ell}^{\alpha\gamma} (y)\bigg\}.
$$
Then
$$
H_{i\ell}^{\alpha\gamma} (y)=\frac{\partial}{\partial y_k}
\bigg\{ F_{i\ell k}^{\alpha\gamma} (y)\bigg\}
$$
and hence
\begin{equation}\label{6.3.5}
\aligned
\left(\frac{\partial w}{\partial \nu_\varep}\right)^\alpha
&=n_i (x) H_{i\ell}^{\alpha\gamma}
\left(\frac{x}{\varep}\right)\\
&=
n_i(x) \frac{\partial}{\partial x_k}
\bigg\{\varep
F_{i\ell k}^{\alpha \gamma} \left(\frac{x}{\varep}\right)\bigg\}.
\endaligned
\end{equation}
We claim that
\begin{equation}\label{6.3.6}
\frac{\partial }{\partial y_i}
\bigg\{ F_{i\ell k}^{\alpha\gamma} (y)\bigg\}
=0.
\end{equation}
Assume the claim is true. We may then write
\begin{equation}\label{6.3.7}
\left(\frac{\partial w}{\partial \nu_\varep}\right)^\alpha
=
n_i(x) \frac{\partial}{\partial x_k}
\bigg\{\varep
F_{i\ell k}^{\alpha \gamma} \left(\frac{x}{\varep}\right)\bigg\}
-
n_k(x) \frac{\partial}{\partial x_i}
\bigg\{\varep
F_{i\ell k}^{\alpha \gamma} \left(\frac{x}{\varep}\right)\bigg\}
\qquad \text{ on }\partial\Omega.
\end{equation}
Since $\| \varep F_{i\ell k}^{\alpha\gamma} (x/\varep)\|_\infty \le C\varep$,
we obtain the desired (\ref{6.3.2}).

Finally, to show (\ref{6.3.6}), we observe that
$$
\frac{\partial}{\partial y_i}
\bigg\{ H_{i\ell}^{\alpha\gamma}(y)\bigg\}=0 \qquad \text{ in }\brd,
$$
which follows directly from (\ref{corrector-equation}).
In view of (\ref{6.3.4}) this implies that
 $\frac{\partial}{\partial y_i}
\big\{U_{i\ell}^{\alpha\gamma} (y)\big\}$ is harmonic in $\brd$.
Since it is also periodic, we may deduce that  
$\frac{\partial}{\partial y_i}
\big\{U_{i\ell}^{\alpha\gamma} (y)\big\}$ is constant.
As a result,
$$
\frac{\partial }{\partial y_i}
\bigg\{ F_{i\ell k}^{\alpha\gamma} (y)\bigg\}
=\frac{\partial^2}{\partial y_k\partial y_i}
\bigg\{ U_{i\ell}^{\alpha \gamma}(y)\bigg\}
=0 \qquad \text{ in }\brd.
$$
This completes the proof of Lemma \ref{6.2-lemma}.
\end{proof}

\noindent{\bf Proof of Theorem \ref{corrector-theorem}.}
It follows from (\ref{Psi}) and (\ref{6.3.1}) that
\begin{equation}\label{6.4.1}
|\nabla \Phi_\varep (x)| \le C +\frac{C\varep}{\delta (x)}
\qquad \text{ for any }x\in \Omega.
\end{equation}
This implies that $|\nabla \Phi_\varep (x)|\le C$ if $\delta(x)\ge c\varep$.
To estimate $|\nabla \Phi_\varep (x)|$ for $x$ with $\delta (x)\le c\varep$,
we use a standard blow-up argument.

Fix $j$ and $\beta$.
Let $w(x)=\varep^{-1}\Phi_{\varep,j}^\beta (\varep x)$.
Then $\mathcal{L}_1 (w)=0$ and
$$
\frac{\partial w}{\partial\nu_1}
=\frac{\partial \Phi_{\varep, j}^\beta}{\partial \nu_\varep}
(\varep x)
=n_i(\varep x) \hat{a}_{ij}^{\alpha\beta}.
$$
Since $\Omega$ is a $C^{1, \alpha_0}$ domain, its normal $n(x)$ is H\"older
continuous.
Thus, by the classical regularity results for
the Neumann problem with data in H\"older spaces,
\begin{equation}\label{6.4.2}
\|\nabla \Phi_\varep\|_{L^\infty (B(Q,\varep)\cap \Omega)}
\le
C+ C\left\{ \frac{1}{\varep^d}
\int_{B(Q,2\varep)\cap\Omega}
|\nabla \Phi_\varep|^p \, dx \right\}^{1/p}
\end{equation}
for any $p>0$, where $Q\in \partial\Omega$ and $C$ depends only on $d$, $m$,
$p$, $\mu$, $\lambda$,
$\tau$ and $\Omega$.
We remark that estimate (\ref{6.4.2}) with $p=2$ is well known and the case $0<p<2$
follows from the case $p=2$ by a convexity argument.
Finally, it follows from (\ref{6.4.1}) and (\ref{6.4.2}) with $p<1$ that
$$
\|\nabla \Phi_\varep\|_{L^\infty(B(Q,\varep)\cap\Omega)}
\le C.
$$
This finishes the proof of Theorem \ref{corrector-theorem}.
\qed

\begin{remark}\label{remark-6.1}
{\rm Fix $\eta\in C_0^\infty(\mathbb{R}^{d-1})$ so that $\eta(x^\prime)=1$ for $|x^\prime|\le 2$
and $\eta(x^\prime)=0$ for $|x^\prime|\ge 3$.
For any function $\psi$ satisfying the condition (\ref{psi}),
we may construct a bounded $C^{1, \alpha_0}$ domain $\Omega_\psi$
in $\brd$ with the following property,
\begin{equation}\label{6.6.1}
\aligned
& D_{\psi\eta} (4)\subset \Omega_\psi\subset \big\{ (x^\prime, x_d): \ |x^\prime|<8
\text{ and } |x_d|<8(M_0+1)\big\},\\
&\big\{ (x^\prime, (\psi\eta)(x^\prime)): \ |x^\prime|<4\big\}
\subset \partial\Omega_\psi.
\endaligned
\end{equation}
Clearly, the domain $\Omega_\psi$ can be constructed in such a way that
$\Omega_\psi\setminus \{ (x^\prime, (\psi\eta)(x^\prime)): \ |x^\prime|\le 4\}$
depends only on $M_0$.

Let $\Phi_\varep(x)=\Phi_\varep(x, \Omega_\psi, A)$ be the matrix of functions satisfying
 (\ref{Phi}) with $\Omega=\Omega_\psi$ and $\Phi_\varep (0)=0$.
It follows from Theorem \ref{corrector-theorem} that
$\|\nabla \Phi_\varep\|_{L^\infty (\Omega)} \le C$, where
$C$ depends only on $d$, $m$, $\mu$, $\lambda$, $\tau$ and $(\alpha_0, M_0)$.
}
\end{remark}

\section{Boundary Lipschitz estimates}

In this section we establish the uniform boundary Lipschitz estimate
under the assumption that $A\in \Lambda(\mu,\lambda,\tau)$ and
$A^*=A$.

\begin{thm}\label{boundary-Lipschitz-theorem}
Let $\Omega$ be a bounded $C^{1,\alpha_0}$ domain.
Suppose that $A\in \Lambda(\mu, \lambda,\tau)$ and $A^*=A$.
Let $\mathcal{L}_\varep (u_\varep)=0$ in $B(Q,\rho)\cap \Omega$
and $\frac{\partial u_\varep}{\partial\nu_\varep}=g
$ on $B(Q,\rho)\cap \partial\Omega$ for some $Q\in \partial\Omega$
and $0<\rho<c$.
Assume that $g\in C^\eta (B(Q, \rho)\cap\partial\Omega)$ for some $\eta\in (0,\alpha_0)$.
Then
\begin{equation}\label{boundary-Lipschitz-estimate}
\| \nabla u_\varep\|_{L^\infty (B(Q,\rho/2)\cap\Omega)}
\le C\left\{
\rho^{-1} \| u_\varep\|_{L^\infty (B(Q, \rho)\cap\Omega)}
+\| g\|_{C^\eta (B(Q, \rho)\cap \partial\Omega)}\right\},
\end{equation}
where $c=c(\Omega)>0$ and
$C$ depends only on $d$, $m$, $\mu$, $\lambda$, $\tau$, $\eta$ and $\Omega$.
\end{thm}

Let $D(\rho)=D(\rho, \psi)$ and $\Delta (\rho)=D(\rho, \psi)$ 
be defined by (\ref{definition-of-D}) with $\psi\in C^{1, \alpha_0}(\br^{d-1})$,
$\psi(0)=|\nabla\psi(0)|=0$ and $\| \nabla \psi\|_{C^{\alpha_0}(\br^{d-1})}\le M_0$.
We will use $\| g\|_{C^{0,\eta}(K)}$ to denote 
$$
\inf\big\{ M: 
\, |g(x)-g(y)|\le M |x-y|^\beta \text{ for all } x,y \in K\big\}.
$$

\begin{lemma}\label{lemma-7.3}
Let $0<\eta<\alpha_0$ and
$\kappa=(1/4)\eta$.
Let $\Phi_\varep =\Phi_\varep (x, \Omega_\psi, A)$
be defined as in Remark \ref{remark-6.1}.
There exist constants $\varep_0>0$, $\theta\in (0,1)$ and $C_0>0$, depending only on 
$d$, $m$, $\mu$, $\lambda$, $\tau$, $\eta$ and $(\alpha_0, M_0)$, such that
\begin{equation}\label{estimate-7.3.1}
 \| u_\varep -<
\Phi_\varep, \mathbf{B}_\varep >\|_{L^\infty (D(\theta))}\\
\le 
\theta^{1+\kappa},
\end{equation}
for some $\mathbf{B}_\varep =(b_{\varep, j}^\beta)\in \mathbb{R}^{dm}$ with 
the property that 
$$
|\mathbf{B}_\varep |\le C_0\theta^{-1}
\| u_\varep\|_{L^\infty (D(\theta))}
 \text{ and } <n(0)\hat{A}, \mathbf{B}_\varep>=
n_i(0)\hat{a}_{ij}^{\alpha\beta} b_{\varep, j}^\beta
=0,
$$ 
whenever
$$
\varep<\varep_0, \quad
\mathcal{L}_\varep (u_\varep)=0 \text{ in } D(1),\quad
\frac{\partial u_\varep}{\partial\nu_\varep} =g \text{ on }
\Delta (1), \quad u_\varep (0)=0,
$$
and
\begin{equation}\label{estimate-7.3.2}
\| g\|_{C^{0,\eta} (\Delta(1))} \le 1, \quad g(0)=0,
\quad 
\| u_\varep\|_{L^\infty (D(1))}\le 1.
\end{equation}
\end{lemma}

\begin{proof}
Let $\mathcal{L}_0 =-\text{\rm div} (A^0\nabla )$, 
where $A^0 =(\hat{a}^{\alpha\beta}_{ij})$
 is a constant $m\times m$ matrix satisfying (\ref{ellipticity}).
By boundary H\"older estimates for gradients of solutions to elliptic systems with constant
coefficients in $C^{1,\alpha_0}$ domains,
\begin{equation}\label{7.3.3}
\aligned
& \| w-<x,(\overline{\nabla w})_{D(r)}> \|_{L^\infty (D(r))}\\
&\qquad
\le C_1 r^{1+2\kappa} \left\{ \| g\|_{C^\eta (\Delta (1/2))}
+
\| w\|_{L^\infty (D(1/2))}\right\},
\endaligned
\end{equation}
for any $r\in (0,1/4)$, whenever $\mathcal{L}_0 (w)=0$ in $D(1/2)$,
$\frac{\partial w}{\partial\nu_0}=g$ on $\Delta(1/2)$ and $w(0)=0$.
The constant $C_1$ in (\ref{7.3.3})
 depends only on $d$, $m$, $\mu$, $\eta$ and $(\alpha_0, M_0)$.
Observe that if 
$$
g(0)=<n(0)A^0, (\nabla w)(0)>=0,
$$
 then
$
\| g\|_{C^\eta (\Delta (1/2))} \le C\| g\|_{C^{0,\eta} (\Delta(1/2))}
$ 
and
\begin{equation}\label{7.3.3.1}
\aligned
& |<n(0)A^0, (\overline{\nabla w})_{D(r)}>|
 =|<n(0)A^0, (\overline{\nabla w})_{D(r)}-(\nabla w) (0)>|\\
&\qquad \le Cr^{2\kappa} \left\{ \| g\|_{C^{0,\eta} (\Delta (1/2))}
+
\| w\|_{L^\infty (D(1/2))}\right\}.
\endaligned
\end{equation}
Consequently, if we let $\mathbf{B}_0 =(b_{0,j}^\beta)\in \mathbb{R}^{dm}$ with
\begin{equation}\label{7.3.3.2}
b_{0,j}^\beta
=\left(\overline{\frac{\partial w^\beta}{\partial x_j}}\right)_{D(r)}
-n_j(0) h^{\beta\gamma} n_i (0) \hat{a}_{i\ell}^{\gamma \alpha} 
\left(\overline{\frac{\partial w^\alpha}{\partial x_\ell}}\right)_{D(r)},
\end{equation}
where $(h^{\alpha\beta})_{m\times m}$ is the inverse matrix
of $(n_i(0)n_j(0)\hat{a}_{ij}^{\alpha\beta})_{m\times m}$, then
\begin{equation}\label{7.3.3.3}
\| w-<x,\mathbf{B}_0 > \|_{L^\infty (D(0,r))}
\le C_2 r^{1+2\kappa},
\end{equation}
for any $r\in (0,1/4)$, provided that $\mathcal{L}_0 (w)=0$ in $D(1/2)$,
$\frac{\partial w}{\partial\nu_0}=g$ on $\Delta(1/2)$, $w(0)=0$,
\begin{equation}\label{7.3.4}
\| g\|_{C^{0,\eta} (\Delta (1/2))}\le 1, \quad g(0)=0
\quad \text{ and } \quad
\| w\|_{L^\infty (D(1/2))}\le 1,
\end{equation}
where $C_2$ depends only on $d$, $m$, $\mu$, $\eta$ and $(\alpha_0, M_0)$.

Next we choose $\theta\in (0,1/4)$ so small that $2C_2\theta^{\kappa}
\le 1$.
We shall show by contradiction that for this $\theta$, there exists
$\varep_0>0$, depending only on $d$, $m$,
$\mu$, $\lambda$, $\tau$, $\eta$ and $(\alpha_0, M_0)$,
such that estimate (\ref{estimate-7.3.1})
holds with
\begin{equation}\label{7.3.3.5}
b_{\varep,j}^\beta
=\left(\overline{\frac{\partial u_\varep^\beta}{\partial x_j}}\right)_{D(\theta)}
-n_j(0) h^{\beta\gamma} n_i (0) \hat{a}_{i\ell}^{\gamma \alpha} 
\left(\overline{\frac{\partial u_\varep^\alpha}{\partial x_\ell}}\right)_{D(\theta)},
\end{equation}
if  $0<\varep<\varep_0$ and $u_\varep$ satisfies the conditions in Lemma \ref{lemma-7.3}.
We recall that 
 $(\hat{a}_{ij}^{\alpha\beta})$ 
in (\ref{7.3.3.5})
is the homogenized matrix given by (\ref{homogenized-coefficient}).
It is easy to verify that $n_i(0) \hat{a}_{ij}^{\alpha\beta} b_{\varep, j}^\beta =0$.
Also, by the divergence theorem, $|\mathbf{B}_\varep|
\le C_0 \theta^{-1} \| u_\varep\|_{L^\infty(D(\theta))}$. 

To show (\ref{estimate-7.3.1}) by contradiction,
let's suppose that there exist sequences
$\{ \varep_k\}$, $\{ A^k\} $, $\{ u_{\varep_k}\}$,  $\{ g_k\}$ and $\psi_k$ such that
$\varep_k\to 0$, $A^k\in \Lambda (\mu, \lambda,\tau)$, $\psi_k$ satisfies (\ref{psi}),
\begin{equation}\label{7.3.5}
\left\{
\aligned
\mathcal{L}_{\varep_k}^k (u_{\varep_k})  & =0 &\quad & \text{ in }D_k(1),\\
\frac{\partial u_{\varep_k}}{\partial\nu_{\varep_k}} & =g_k& \quad & \text{ on } \Delta_k(1),\\
u_{\varep_k} (0) &=g_k (0)  =0,
\endaligned
\right.
\end{equation}
\begin{equation}\label{7.3.6}
\| g_k\|_{C^{0,\eta} (\Delta_k (1))}
\le 1, \qquad \text{ \ }\qquad \| u_{\varep_k}\|_{L^\infty(D_k(1))}
\le 1,
\end{equation}
and
\begin{equation}\label{7.3.7}
\| u_{\varep_k} -<\Phi_{\varep_k}^k, \mathbf{B}^k_\varep>
\|_{L^\infty (D_k (\theta))}
> \theta^{1+\kappa},
\end{equation}
where $D_k (r) =D(r, \psi_k)$,
$\Delta_k (r)=\Delta (r, \psi_k)$,
 $\Phi^k_{\varep_k}=\Phi_{\varep_k} (x, \Omega_{\psi_k}, A^k)$
and $\mathbf{B}_\varep^k$ is given by (\ref{7.3.3.5}).
By passing to subsequences we may assume that as $k\to \infty$,
\begin{equation}\label{7.3.8}
\aligned
 \hat{A}^k  & \to A^0,\\
\psi_k  & \to \psi_0 \quad \text{ in } C^1 (|x^\prime|< 4),\\
 g_k(x^\prime, \psi_k(x^\prime))
& \to g_0 (x^\prime, \psi_0 (x^\prime)) \quad
\text{ in } C(|x^\prime|<1).\\
\endaligned
\end{equation}
Since $\| u_{\varep_k}\|_{C^\eta (D(1/2,\psi_k))} +\| \Phi_{\varep_k}^k\|_{C^\eta
(D(1/2, \psi_k))}\le C$ by
Theorem \ref{boundary-holder-theorem}, again by passing
to subsequences, we may also assume that
\begin{equation}\label{7.3.9}
\aligned
& u_{\varep_k} (x^\prime, x_d-\psi_k(x^\prime))
\to u_0 (x^\prime, x_d-\psi_0 (x^\prime))
\quad \text{ uniformly on } D(1/2, 0),\\
& R_{\varep_k}^k (x^\prime, x_d-\psi_k (x^\prime))
\quad\text{ converges uniformly on } D(1/2, 0),
\endaligned
\end{equation}
where $R_{\varep_k}^k (x) =\Phi_{\varep_k}^k (x) -x$.
Furthermore, in view of Theorem \ref{compactness-theorem}, we may assume that
$\mathcal{L}_0 (u_0)=0$ in $D(1/2, \psi_0)$ and $\frac{\partial u_0}{\partial\nu_0}=
g_0$ on $\Delta (1/2, \psi_0)$, where $\mathcal{L}_0
=-\text{\rm div} (A^0\nabla )$.

Note that by Lemma \ref{6.3-lemma},
 $R_{\varep_k}^k (x^\prime, x_d-\psi_k (x^\prime))$ must converge to a constant.
Since $R_{\varep_k}^k (0)=0$, we deduce that
$R_{\varep_k}^k (x^\prime, x_d-\psi_k (x^\prime))$ converges uniformly to $0$ on $D(1/2,0)$.
Thus, in view of (\ref{7.3.6})-(\ref{7.3.9}), we may conclude that
$ u_0(0)=g(0)=0$,
\begin{equation}\label{7.3.10}
 \| g\|_{C^{0,\eta}(\Delta(1/2, \psi_0))} \le 1, \quad \quad
\| u_0\|_{L^\infty(D(1/2, \psi_0))} \le 1
\end{equation}
and
\begin{equation}\label{7.3.11}
\| u_0 -<x, \mathbf{B}_0>
\|_{L^\infty (D (\theta, \psi_0))}
\ge \theta^{1+\kappa}.
\end{equation}
This, however, contradicts with (\ref{7.3.3.3})-(\ref{7.3.4}).
\end{proof}

\begin{remark}\label{remark-7.3}
{\rm 
Let $w =<\Phi_\varep ,\mathbf{B}_\varep>
=\Phi_{\varep, j}^{\alpha\beta} (x) b_{\varep,j}^\beta$,
 where $\Phi_\varep$ and $\mathbf{B}_\varep$ are given
by Lemma \ref{lemma-7.3}. Then $\mathcal{L}_\varep (w)=0$ and $\frac{\partial w}
{\partial\nu_\varep} =n_i(x)\hat{a}_{ij}^{\alpha\beta}b_{\varep, j}^\beta$.
In particular, we have $w(0)=0$ and
$\frac{\partial w}{\partial\nu_\varep} (0)=0$.
Also, note that in Lemma \ref{lemma-7.3},
one may choose any $\theta\in (0,\theta_1)$, where $2C_2\theta_1^\kappa=1$.
These observations are important to the proof of the next lemma.
}
\end{remark}

\begin{lemma}\label{lemma-7.4}
Let $\kappa$, $\varep_0$, $\theta$ be the constants
given by Lemma \ref{lemma-7.3}.
Suppose that $\mathcal{L}_\varep (u_\varep) =0$ in $D(1, \psi)$,
$\frac{\partial u_\varep}{\partial\nu_\varep} =g$ on $\Delta(1, \psi)$
and $u_\varep (0)=g(0)=0$.
Assume that $\varep<\theta^{\ell -1}\varep_0$ for some $\ell\ge 1$.
Then there exist $\mathbf{B}_\varep^j \in \mathbb{R}^{dm}$ for $j=0, 1, \dots, \ell-1$, such that
$$
<n(0)\hat{A}, \mathbf{B}_\varep^j>=0, \quad
|\mathbf{B}_\varep^j|\le C J 
$$
and
\begin{equation}\label{7.4.2}
\| u_\varep -\sum_{j=0}^{\ell -1}
\theta^{\kappa j}
< \Pi_\varep^j, \mathbf{B}_\varep^j >\|_{L^\infty(D(\theta^\ell, \psi))} 
\le \theta^{\ell(1+\kappa)} J,
\end{equation}
where 
$$
\aligned
& \Pi_\varep^j(x) =\theta^{j}\Phi_{\frac{\varep}{\theta^j}}
 (\theta^{-j}x, \Omega_{\psi_j}, A),\\ 
& J=\max \left\{
\| g\|_{C^{0,\eta} (\Delta(1, \psi))}, \| u_\varep\|_{L^\infty(D(1, \psi))}\right\}
\endaligned
$$
and $\psi_j (x^\prime)
=\theta^{-j}\psi(\theta^j x^\prime)$.
\end{lemma}

\begin{proof}
The lemma is proved by an induction argument on $\ell$.
The case  $\ell=1$ follows  by applying Lemma \ref{lemma-7.3} to $u_\varep/J$.
Suppose now that Lemma \ref{lemma-7.4} holds for some $\ell\ge 1$.
Let $\varep<\theta^\ell \varep_0$.
Consider the function
$$
w(x)=\theta^{-\ell}
\left\{ u_\varep (\theta^\ell x)-\sum_{j=0}^{\ell-1} 
\theta^{\kappa j} <\Pi_\varep^j (\theta^\ell x), \mathbf{B}_\varep^j >\right\}
$$
on $D(1, \psi_\ell)$. Note that
 $\mathcal{L}_{\frac{\varep}{\theta^\ell}} (w)=0$ in $D(1, \psi_\ell)$,
$w(0)=0$ and by the induction assumption,
\begin{equation}\label{7.4.2.1}
\| w\|_{L^\infty (D(1, \psi_\ell))} 
\le \theta^{\ell \kappa}J.
\end{equation}
Let 
$$
h(x)=\frac{\partial w}{\partial \nu_{\frac{\varep}{\theta^\ell}}} (x)
\qquad \text{ on } \Delta (1, \psi_\ell).
$$
Then
\begin{equation}\label{7.4.3}
h(x)=g(\theta^\ell x)
-\sum_{j=1}^{\ell-1}
\theta^{\kappa j} < n(\theta^\ell x ) \hat{A}, \mathbf{B}_\varep^j>,
\end{equation}
where $n$ denotes the unit outward normal to $\Delta(1, \psi)$.
It follows that $h(0)=0$. Since $\varep \theta^{-\ell}<\varep_0$,
we may then apply the estimate for the case $\ell=1$ to obtain 
\begin{equation}\label{7.4.4}
\aligned
&
\| w-<\Phi_{\frac{\varep}{\theta^\ell}} (x, \Omega_{\psi_\ell}, A), 
\mathbf{B}_{\frac{\varep}{\theta^\ell}}>\|_{L^\infty (D(\theta, \psi_\ell))}\\
&\qquad
\le \theta^{1+\kappa}
\max \left\{ \| h\|_{C^{0,\eta}(\Delta (1, \psi_\ell))},
\| w\|_{L^\infty (D(1, \psi_\ell))}\right\},
\endaligned
\end{equation}
where $\mathbf{B}_{\frac{\varep}{\theta^\ell}}\in \mathbb{R}^{dm}$ satisfies the conditions
$<n(0)\hat{A}, \mathbf{B}_{\frac{\varep}{\theta^\ell}}>=0$ and
\begin{equation}\label{7.4.4.1}
|\mathbf{B}_{\frac{\varep}{\theta^\ell}}|\le C\max
\left\{ \| h\|_{C^{0,\eta}(\Delta(1, \psi_\ell))},
\| w\|_{L^\infty(D(1, \psi_\ell))}\right\}.
\end{equation}
It follows that
\begin{equation}\label{7.4.5}
\aligned
&
\| u_\varep (x)-\sum_{j=0}^{\ell-1}
\theta^{\kappa j} <\Pi_\varep^j (x), \mathbf{B}_\varep^j>
-\theta^\ell <\Phi_{\frac{\varep}{\theta^\ell}} (\theta^{-\ell} x, \Omega_{\psi_\ell}, A),
\mathbf{B}_{\frac{\varep}{\theta^\ell}}>\|_{L^\infty (D(\theta^{\ell +1}, \psi))}\\
&\qquad
\le \theta^{\ell +1+\kappa}
\max \left\{ \| h\|_{C^{0,\eta}(\Delta(1, \psi_\ell))},
\| w\|_{L^\infty (D(1, \psi_\ell))}\right\}.
\endaligned
\end{equation}

To estimate the right hand side of (\ref{7.4.5}), we observe that
$$
\aligned
\| h\|_{C^{0,\eta}(\Delta (1, \psi_\ell))}
&\le \theta^{\ell \eta} \| g\|_{C^{0,\eta} (\Delta(1, \psi))}
+\sum_{j=0}^{\ell-1} \theta^{\kappa j} \cdot CJ \cdot
\theta^{\ell \eta} \| n\|_{C^{0, \eta}(\Delta(1, \psi))}\\
& \le
\theta^{4\ell \kappa} J \left\{ 1+\frac{C\| n\|_{C^{0,\eta}(\Delta(1, \psi))}}{1-\theta^\kappa}\right\},
\endaligned
$$
since $\eta=4\kappa$.
Since $0<\eta<\alpha_0$,
by making an initial dilation of $x$, if necessary, we may assume that 
$\|n\|_{C^{0,\eta}(\Delta (1,\psi))}$ is small so that
\begin{equation}\label{7.4.6}
\theta^{\kappa}
\left\{ 1+\frac{C\| n\|_{C^{0,\eta}(\Delta(1, \psi))}}{1-\theta^\kappa}\right\}
\le 1.
\end{equation}
This implies that
\begin{equation}\label{7.4.7}
\| h\|_{C^{0,\eta}(\Delta (1, \psi_\ell))}
\le \theta^{\ell \kappa} J.
\end{equation}
This, together with (\ref{7.4.2.1}) and (\ref{7.4.5}), gives
\begin{equation}\label{7.4.8}
\| u_\varep -\sum_{j=0}^{\ell}
\theta^{\kappa j} <\Pi_\varep^j, \mathbf{B}_\varep^j>
\|_{L^\infty (D(\theta^{\ell +1}, \psi))}\\
\le \theta^{(\ell +1)(1+\kappa)} J,
\end{equation}
where we have chosen
 $\mathbf{B}_\varep^\ell =\theta^{-\ell\kappa} \mathbf{B}_{\frac{\varep}{\theta^\ell}}$.
Finally, 
in view of (\ref{7.4.4.1}), (\ref{7.4.2.1}) and (\ref{7.4.7}), we have
$|\mathbf{B}_\varep^\ell |\le C J$.
This completes the induction argument.
\end{proof}

\begin{lemma}\label{lemma-7.6}
Suppose that $\mathcal{L}_\varep (u_\varep) =0$ in $D(1)$
and $\frac{\partial u_\varep}{\partial\nu_\varep} =g$
on $\Delta(1)$.
Then
\begin{equation}\label{7.6.1}
\int_{D(\rho)}
|\nabla u_\varep|^2\, dx
\le C \rho^d \left\{ \| u_\varep\|_{L^\infty (D(1))}^2
+\| g\|_{C^\eta(\Delta(1))}^2 \right\},
\end{equation}
for any $0<\rho<(1/2)$, where $C$ depends only on $\mu$, $\lambda$,
$\tau$, $\eta$ and $(M_0, \alpha_0)$.
\end{lemma}

\begin{proof}
By subtracting a constant we may assume that $u_\varep (0)=0$.
We may also assume that $g(0)=0$. To see this, consider
$$
v^\alpha_\varep (x)
=u_\varep^\alpha (x)-\Phi_{\varep, j}^{\alpha\beta} (x) n_j(0) b^\beta,
$$
where $(b^\beta)\in \br^{m}$ solves the linear system 
$n_i(0)n_j(0)\hat{a}^{\alpha\beta}_{ij} b^\beta
=g^\alpha (0)$.
Then $\mathcal{L}_\varep (v_\varep)=0$ in $D(1)$, $v_\varep (0)=0$ and
$$
\left(\frac{\partial v_\varep}{\partial\nu_\varep}\right)^\alpha (x)
=g^\alpha (x) -n_i(x)\hat{a}_{ij}^{\alpha\beta} n_j(0)b^\beta
\qquad \text{ on } \Delta (1).
$$
Thus $\frac{\partial v_\varep}{\partial\nu_\varep} (0)=0$.
Since $\|\Phi_\varep \|_{L^\infty(D(1))}
+\|\nabla \Phi_\varep\|_{L^\infty(D(1))}\le C$, the desired estimate for $u_\varep$ 
follows from the corresponding estimate for $v_\varep$.

Under the assumption that $u_\varep (0)=g(0)=0$, we will show that
\begin{equation}\label{7.6.2}
\| u_\varep\|_{L^\infty(D(\rho))}
\le C 
\rho \left\{ \| u_\varep\|_{L^\infty (D(1))}
+\| g\|_{C^\eta(\Delta(1))} \right\},
\end{equation}
for any $0<\rho<(1/2)$.
Estimate (\ref{7.6.1}) follows from (\ref{7.6.2}) 
by Cacciopoli's inequality (\ref{Cacciopoli}).

Let $\kappa$, $\varep_0$, $\theta$ be the constants given by Lemma \ref{lemma-7.3}.
Let $0<\varep<\theta\varep_0$ (the case $\varep\ge \theta\varep_0$
follows from the classical regularity estimates).
Suppose that
$$
\theta^{i+1} \le \frac{\varep}{\varep_0} <\theta^i
\qquad \text { for some } i\ge 1.
$$
Let $\rho\in (0, 1/2)$. We first consider the case $\frac{\varep}{\varep_0}
\le \rho<\theta$. 
Then $\theta^{\ell+1}\le \rho<\rho^\ell$ for some $\ell=1,\dots, i$.
It follows that
\begin{equation}\label{7.6.3}
\aligned
&\| u_\varep\|_{L^\infty (D(\rho))}
 \le \| u_\varep\|_{L^\infty (D(\theta^\ell))}\\
& \le \| u_\varep -\sum_{j=0}^{\ell-1} \theta^{\kappa j}
<\Pi_\varep^j, \mathbf{B}_\varep^j >\|_{L^\infty (D(\theta^\ell))}
+\sum_{j=0}^{\ell-1} \theta^{\kappa j} |\mathbf{B}_\varep^j|
\| \Pi_\varep^j\|_{L^\infty (D(\theta^\ell))}\\
&\le \theta^{\ell (1+\kappa)} J
+CJ \sum_{j=0}^{\ell-1} \theta^{\kappa j}\|\Pi_\varep^j\|_{L^\infty (D(\theta^\ell))},
\endaligned
\end{equation}
where $J=\max \big\{ \| g\|_{C^{0,\eta}(D(1))}, \| u_\varep\|_{L^\infty(D(1))}\big\}$ and
we have used Lemma \ref{lemma-7.4}.
Recall that $\Pi_\varep^j (x)=\theta^j \Phi_{\frac{\varep}{\theta^j}} (\theta^{-j} x,
\Omega_{\psi_j}, A)$.
By Remark \ref{remark-6.1} we have
$\Pi_\varep^j (0)=0$ and $\|\nabla \Pi_\varep^j \|_{L^\infty(D(1))}\le C$. Hence,
$$
\| \Pi_\varep^j \|_{L^\infty(D(\theta^\ell))} \le C\theta^\ell.
$$
This, together with (\ref{7.6.3}), gives $\| u_\varep\|_{L^\infty(D(\rho))}
\le C\rho J$ for any $\frac{\varep}{\varep_0}\le \rho<\frac12$ (the case $\theta\le \rho<(1/2)$ is 
trivial).

To treat the case $0<\rho<\frac{\varep}{\varep_0}$,  we use a blow-up argument.
Let $w(x)=\varep^{-1} u_\varep (\varep x)$.
Then $\mathcal{L}_1 (w) =0$ in $D(2\varep_0^{-1}, \psi_\varep)$ and
$\frac{\partial w}{\partial\nu_1} (x) =g(\varep x)$ on
$\Delta (2\varep_0^{-1}, \psi_\varep)$,
where $\psi_\varep (x^\prime) =\varep^{-1}\psi(\varep x^\prime)$.
By the classical regularity estimate,
$$
\|\nabla w\|_{L^\infty (D(\frac{1}{\varep_0}, \psi_\varep))}
\le C\left\{
\| w\|_{L^\infty(D(\frac{2}{\varep_0}, \psi_\varep))}
+\|\frac{\partial w}{\partial\nu_1}\|_{C^\eta (\Delta (\frac{2}{\varep_0}, \psi_\varep))}\right\}.
$$
It follows that
$$
\|\nabla u_\varep\|_{L^\infty (D(\frac{\varep}{\varep_0}))}
 \le C
\left\{
\varep^{-1}
\| u\|_{L^\infty(D(\frac{2\varep}{\varep_0}))}
+\| g\|_{C^\eta (\Delta(1))}\right\}\\
\le C J,
$$
where we have used the estimate (\ref{7.6.2}) with $\rho=\frac{2\varep}{\varep_0}$
for the last inequality.
Finally, since $u_\varep (0)=0$, for $0<\rho<\frac{\varep}{\varep_0}$, we obtain
$$
\| u_\varep\|_{L^\infty (D(\rho))}
\le C\rho \|\nabla u_\varep\|_{L^\infty(D(\frac{\varep}{\varep_0}))}
\le C\rho J.
$$
This completes the proof of (\ref{7.6.2}).
\end{proof}

\noindent{\bf Proof of Theorem \ref{boundary-Lipschitz-theorem}.}
By rescaling we may assume that $\rho=1$.
By a change of the coordinate system, we may deduce from Lemma \ref{lemma-7.6} that
if $P\in \partial\Omega$, $|P-Q|<\frac12$ and $0<r<\frac14$,
$$
\int_{B(P, r)\cap\Omega}
|\nabla u_\varep|^2\, dx
\le C r^d
\left\{ \| u_\varep\|^2_{L^\infty (B(Q,1)\cap \Omega)}
+\| g\|^2_{C^\eta(B(Q, 1)\cap\partial\Omega)}\right\},
$$
where $C$ depends only on $d$, $m$, $\mu$, $\lambda$, $\tau$, $\eta$ and $\Omega$.
This, together with the interior estimate (\ref{interior-estimate}),
implies that
$$
\|\nabla u_\varep\|_{L^\infty(B(Q,\frac12)\cap \Omega)}
\le C \left\{ 
\| u_\varep\|_{L^\infty (B(Q,1)\cap \Omega)}
+\| g\|_{C^\eta(B(Q, 1)\cap\partial\Omega)}\right\}.
$$
The proof of Theorem \ref{boundary-Lipschitz-theorem} is now complete.
\qed

\section{Proof of Theorem \ref{Lipschitz-estimate-theorem}}

Under the condition $A\in \Lambda (\lambda, \mu, \tau)$,
we have proved in Section 5 that
\begin{equation}\label{8.1}
 |N_\varep (x,y)|\le \frac{C}{|x-y|^{d-2}} \qquad  \text{ if } d\ge 3.
\end{equation}
With the additional assumption $A^*=A$, we may use Theorem \ref{boundary-Lipschitz-theorem}
to show that for $d\ge 3$,
\begin{equation}\label{8.2}
\aligned
 |\nabla_x N_\varep (x,y)|+|\nabla_y N_\varep (x,y)|
 & \le \frac{C}{|x-y|^{d-1}},\\
|\nabla_x\nabla_y N_\varep (x,y)| &\le \frac{C}{|x-y|^d}.
\endaligned 
\end{equation}
If $d=2$, one obtains $|N_\varep (x,,y)|\le C_\gamma |x-y|^{-\gamma}$
and $|\nabla_x N_\varep (x,y)| +|\nabla_y N_\varep (x,y)|\le
C_\gamma |x-y|^{-1-\gamma}$ for any $\gamma>0$ (this is not sharp, but sufficient for
the proof of Theorem \ref{Lipschitz-estimate-theorem}).
Now, given $F\in L^q(\Omega)$ for some $q>d$, let
$$
v_\varep (x)=\int_\Omega N_\varep (x,y) F(y)\, dy.
$$
Then $\mathcal{L}_\varep (v_\varep)=F$ in $\Omega$ and 
$\frac{\partial v_\varep}{\partial\nu_\varep}=-\frac{1}{|\partial\Omega|}\int_\Omega F$
on $\partial\Omega$.
Furthermore, it follows from pointwise estimates on $|\nabla_x N_\varep (x,y)|$ that
$\|\nabla v_\varep\|_{L^\infty(\Omega)}\le C \, \| F\|_{L^q(\Omega)}$.
Thus, by subtracting $v_\varep$ from $u_\varep$, we may assume that $F=0$ 
in Theorem \ref{Lipschitz-estimate-theorem}.
In this case we may deduce
from Theorems \ref{boundary-Lipschitz-theorem} and \ref{boundary-holder-theorem} that
for $Q\in \partial\Omega$,
\begin{equation}\label{8.3}
\|\nabla u_\varep\|_{L^\infty (B(Q, \rho/2)\cap \Omega)}
\le C\left\{
\left(\average_{B(Q,\rho)\cap\Omega}|\nabla u_\varep|^2\right)^{1/2}
+\| g\|_{C^\eta (\Delta (Q, \rho))}\right\},
\end{equation}
where $C$ depends only on $d$, $m$,
$\mu$, $\lambda$, $\tau$, $\eta$ and $\Omega$.
Since $\|\nabla u_\varep\|_{L^2(\Omega)} \le C\| g\|_{L^2(\partial\Omega)}$,
the estimate $\| \nabla u_\varep\|_{L^\infty(\Omega)}
\le C\| g\|_{C^\eta (\partial\Omega)}$
 follows from (\ref{8.3}) and the interior estimate (\ref{interior-estimate})
by a covering argument.

\section{Proof of Theorem \ref{maximal-function-theorem}}

As we mentioned in Section 1, the case $p=2$ is proved in \cite{Kenig-Shen-2}
(for Lipschitz domains).
To handle the case $p>2$, we need the following weak reverse H\"older inequality.

\begin{lemma}\label{lemma-9.1}
Let $\Omega$ be a bounded $C^{1,\alpha_0}$ domain.
Suppose that $A\in \Lambda(\lambda, \mu, \tau)$ and $A^*=A$.
Then, for $Q\in \partial\Omega$ and $0<r<r_0$,
\begin{equation}\label{9.1.1}
\sup_{B(Q,r)\cap\partial\Omega}
(\nabla u_\varep)^*
\le C\left\{ \average_{B(Q,2r)\cap\partial\Omega} |(\nabla u_\varep)^*|^2\, d\sigma\right\}^{1/2} ,
\end{equation}
where $u_\varep\in W^{1,2}(B(Q,3r)\cap\Omega)$
is a weak solution to $\mathcal{L}_\varep (u_\varep)=0$ in $B(Q,3r)\cap\Omega$
with either $\frac{\partial u_\varep}{\partial\nu_\varep}=0$ or $u_\varep=0$
on $B(Q,3r)\cap \partial\Omega$.
\end{lemma}

\begin{proof}
Recall that the nontangential maximal function of $(\nabla u_\varep)^*$ is defined by
$$
(\nabla u_\varep)^*(P)
=\sup \big\{
|\nabla u_\varep (x)|:\
x\in\Omega \text{ and } |x-P|< C_0 \, \text{dist}(x, \partial\Omega)\big\},
$$
for $P\in \partial\Omega$,
where $C_0=C(\Omega)>1$ is sufficiently large.
Note that $$
(\nabla u_\varep)^*(P)= \max\big\{ 
\mathcal{M}_{r,1} (\nabla u_\varep), \mathcal{M}_{r,2} (\nabla u_\varep)\big\},
$$ 
where
$$
\aligned
& \mathcal{M}_{r,1} (\nabla u_\varep) (P)
=\sup \big\{
|\nabla u_\varep (x)|:\
x\in\Omega, \ |x-P|\le c_0r \text{ and } |x-P|< C_0 \, \text{dist}(x, \partial\Omega)\big\},\\
& \mathcal{M}_{r,2} (\nabla u_\varep) (P)
=\sup \big\{
|\nabla u_\varep (x)|:\
x\in\Omega, \ |x-P|>c_0r \text{ and } |x-P|< C_0 \, \text{dist}(x, \partial\Omega)\big\},
\endaligned
$$
and $c_0=c(\Omega)>0$ is sufficiently small.
Using interior estimate (\ref{interior-estimate}), it is easy to see that
$\sup_{B(Q,r)\cap\partial\Omega} \mathcal{M}_{r,2} (\nabla u_\varep)$ is 
bounded by the right hand side of (\ref{9.1.1}).  
To estimate $\mathcal{M}_{r,1}(\nabla u_\varep)$,
we observe that
\begin{equation}\label{9.1.2}
\aligned
\sup_{B(Q,r)\cap\partial\Omega} \mathcal{M}_{r,1} (\nabla u_\varep)
& \le \sup_{B(Q,3r/2)\cap\Omega} |\nabla u_\varep|\\
&\le C\, \left\{ \average_{B(Q,2r)\cap\Omega} |\nabla u_\varep|^2\, dx \right\}^{1/2} \\
&\le C\, \left\{ \average_{B(Q,2r)\cap\partial\Omega}
|(\nabla u_\varep)^*|^2\, d\sigma\right\}^{1/2}.
\endaligned
\end{equation}
We point out that the second inequality in (\ref{9.1.2}) follows from 
the boundary Lipschitz estimate. For Neumann condition $\frac{\partial u_\varep}
{\partial\nu_\varep}=0$ on $B(Q,3r)\cap\Omega$, the estimate was given
by Theorem \ref{boundary-Lipschitz-theorem}, while the case of Dirichlet condition 
follows from Theorem 2 in \cite[p.805]{AL-1987}.
\end{proof}

\begin{lemma}\label{lemma-9.2} 
Suppose that $A\in\Lambda (\lambda, \mu, \tau)$ and $A^*=A$.
Let $p>2$ and $\Omega$ be a bounded Lipschitz domain.
Assume that
\begin{equation}\label{9.2.1}
\left(\average_{B(Q,r)\cap\partial\Omega} |(\nabla u_\varep)^*|^p\, d\sigma\right)^{1/p}
\le
C\, \left(\average_{B(Q,2r)\cap\partial\Omega} |(\nabla u_\varep)^*|^2
\, d\sigma\right)^{1/2},
\end{equation}
whenever $u_\varep \in W^{1,2}(B(Q,3r)\cap \Omega)$
is a weak solution to $\mathcal{L}_\varep (u_\varep)=0$ in $B(Q, 3r)\cap\Omega$
and $\frac{\partial u_\varep}{\partial\nu_\varep} =0$ on $B(Q,3r)\cap\partial\Omega$
for some $Q\in \partial\Omega$ and $0<r<r_0$.
Then the weak solutions to $\mathcal{L}_\varep(u_\varep)=0$ in $\Omega$
and $\frac{\partial u_\varep}{\partial\nu_\varep}=g\in L^p(\partial\Omega)$
satisfy the estimate $\|(\nabla u_\varep)^*\|_{L^p(\partial\Omega)}
 \le C\, \| g\|_{L^p(\partial\Omega)}$.
\end{lemma}

\begin{proof} This follows by a real variable argument originating
 in \cite{Caffarelli-1998} and
further developed in \cite{Shen-2005-bounds, Shen-2006-ne, Shen-2007-boundary}.
In \cite{Kim-Shen} the argument was used to prove that for any given $p>2$ and
Lipschitz domain $\Omega$, the solvability of the Neumann problem for Laplace's equation
$\Delta u=0$ in $\Omega$ with $L^p$ boundary data is equivalent to a weak reverse H\"older inequality,
similar to (\ref{9.2.1}).
With the solvability of the $L^2$ Neumann problem for $\mathcal{L}_\varep (u_\varep)
=0$ \cite{Kenig-Shen-2} and interior estimate (\ref{interior-estimate}), the proof of the sufficiency
of the weak reverse H\"older inequality in \cite[pp.1819-1821]{Kim-Shen} extends directly
to the present case. We omit the details.
\end{proof}

It follows from Lemmas \ref{lemma-9.1} and \ref{lemma-9.2} that
Theorem \ref{maximal-function-theorem} holds for $p>2$.
To handle the case $1<p<2$, as in the case of
Laplace's equation \cite{Dahlberg-Kenig-1987},
one considers the solutions of the $L^2$ Neumann problem
with atomic data $\frac{\partial u_\varep}{\partial \nu_\varep} =a$, where
$\int_{\partial\Omega} a=0$,
supp$(a)\subset B(Q,r)\cap\partial\Omega$ for some $Q\in \partial\Omega$ and
$0<r<r_0$, and $\|a\|_{L^\infty(\partial\Omega)}\le r^{1-d}$.
One needs to show that
\begin{equation}\label{9.3.1}
\int_{\partial\Omega} (\nabla u_\varep)^*\, d\sigma \le C.
\end{equation}
The case $1<p<2$ follows from (\ref{9.3.1}) by interpolation.

To prove (\ref{9.3.1}), one first uses the H\"older inequality and the $L^2$
estimate $\|(\nabla u_\varep)^*\|_{L^2(\partial\Omega)}\le C\, \| a\|_{L^2(\partial\Omega)}
\le C r^{\frac{1-d}{2}}$ to see that 
\begin{equation}\label{9.3.2}
\int_{B(Q, Cr)\cap\partial\Omega} (\nabla u_\varep)^*\, d\sigma \le C.
\end{equation}
Next, to estimate $(\nabla u)^*$ on $\partial\Omega\setminus B(Q, Cr)$,
we show that
\begin{equation}\label{9.3.3}
\int_{B(P_0, c\rho)\cap\partial\Omega}
(\nabla u_\varep)^*\, d\sigma \le C\left(\frac{r}{\rho}\right)^\gamma,
\end{equation}
for some $\gamma>0$, where $\rho=|P_0-Q|\ge Cr$. Note that
\begin{equation}\label{9.3.4}
u_\varep (x)
=b +\int_{B(Q,r)\cap\partial\Omega}
\big\{ N_\varep (x,y)-N_\varep (x,Q)\big\} a(y)\, d\sigma (y)
\end{equation}
for some $b\in \br^m$. It follows that
\begin{equation}\label{9.3.5}
|\nabla u_\varep (x)|
\le C
\average_{B(Q,r)\cap\partial\Omega}
\big| \nabla_x \big\{ N_\varep (x,y)-N_\varep (x,Q)\big\}\big|\, d\sigma (y).
\end{equation}
Hence, if $z\in \Omega$ and 
$c\rho\le |z-P|<C_0\delta(z)$ for some $P\in B(P_0, c\rho)\cap
\partial\Omega$, 
$$
\aligned
|\nabla u_\varep (z)|
&\le C\left(\average_{B(z, c\delta(z))} |\nabla u(x)|^2\, dx\right)^{1/2}\\
& \le C \average_{B(Q,r)\cap\partial\Omega}
\left(\average_{B(z,c\delta (z))} |\nabla_x \big\{
N_\varep (x,y)-N_\varep (x,Q)\big\}|^2\, dx \right)^{1/2} d\sigma (y)\\
&\le C\rho^{1-d} \left(\frac{r}{\rho}\right)^\gamma,
\endaligned
$$
where $\delta (z)=\text{dist}(z, \partial\Omega)$
and
we have used the interior estimate, Minkowski's inequality and Theorem \ref{Neumann-theorem-5.2}.
This implies that
\begin{equation}\label{9.3.6}
\int_{B(P_0, c\rho)\cap \partial\Omega}
\mathcal{M}_{2, \rho} (\nabla u_\varep)\, d\sigma \le C\left(\frac{r}{\rho}\right)^\gamma.
\end{equation}

Finally, to estimate $\mathcal{M}_{1,\rho} (\nabla u_\varep)$, we note that the $L^2$ nontangential
maximal function estimate, together with an integration argument, gives
\begin{equation}\label{9.3.7}
\int_{B(P_0, c\rho)\cap\partial\Omega}
|\mathcal{M}_{1,\rho} (\nabla u_\varep)|^2\, d\sigma
\le \frac{C}{\rho}
\int_{B(P_0, 2c\rho)\cap\Omega} |\nabla u_\varep|^2\, dx,
\end{equation}
(see \cite{Dahlberg-Kenig-1987} for the case of Laplace's equation).
It follows by H\"older inequality that
\begin{equation}\label{9.3.8}
\aligned
\int_{B(P_0, c\rho)\cap\partial\Omega}
\mathcal{M}_{1,\rho} (\nabla u_\varep)\, d\sigma
& \le C\rho^{d-1} \left(\average_{B(P_0, 2c\rho)\cap\Omega} |\nabla u_\varep|^2\, dx\right)^{1/2}\\
&\le C\left(\frac{r}{\rho}\right)^\gamma,
\endaligned
\end{equation}
where the last inequality follows from (\ref{9.3.5}) and Theorem \ref{Neumann-theorem-5.2}.
In view of (\ref{9.3.6}) and (\ref{9.3.8}), we have proved (\ref{9.3.2}). The desired 
estimate 
$$
\int_{\partial\Omega\setminus B(Q, Cr)} 
(\nabla u_\varep)^*\, d\sigma \le C
$$
follows from (\ref{9.3.2}) by a simple covering argument.
This completes the proof of (\ref{9.3.1}) and hence of Theorem \ref{maximal-function-theorem}.
\qed

\begin{remark}\label{remark-9.0}
{\rm 
The estimate $\|\nabla u_\varep\|_{L^q(\Omega)} \le C \| g\|_{L^p(\partial\Omega)}$
with $q=\frac{pd}{d-1}$ in Theorem \ref{maximal-function-theorem}
follows from Theorem \ref{W-1-p-theorem},
using the fact that $L^p(\partial\Omega)\subset B^{-\frac{1}{q}, q}(\partial\Omega)$.
The estimate also follows from the observation that $\|w\|_{L^q(\Omega)}
\le C\| (w)^*\|_{L^p(\partial\Omega)}$ for any $w$ 
in a Lipschitz domain $\Omega$. To see this, we note that
\begin{equation}\label{9.0.1}
|w(x)|\le C\int_{\partial\Omega} 
\frac{ (w)^* (Q)}{|x-Q|^{d-1}}\, d\sigma (Q).
\end{equation}
By a duality argument, it then suffices to show that the operator
$$
I_1( f) (x) =\int_{\Omega} \frac{ f(y)}{|x-y|^{d-1}} \, dy
$$ is bounded from $L^{q^\prime}(\Omega)$ to $L^{p^\prime}(\partial\Omega)$.
This may be proved by using fractional and singular integral estimates
(see e.g. \cite[p.712]{Shen-2006-ne}).
}
\end{remark}

\begin{remark}\label{remark-9.1}
{\rm Suppose that $d \ge 3$. For $g\in L^p(\partial\Omega)$,
consider the $L^p$ Neumann problem in the exterior domain
$\Omega_-=\br^d\setminus \overline{\Omega}$,
\begin{equation}\label{exterior-Neumann}
\left\{\aligned
\mathcal{L}_\varep (u_\varep) & =0 \quad \text{ in } \Omega_-,\\
\frac{\partial u_\varep}{\partial \nu_\varep} & =g \quad \text{ on } \partial\Omega,\\
(\nabla u_\varep)^* & \in L^p(\partial\Omega) \text{ and }
u_\varep (x)  =O(|x|^{2-d}) \quad \text{ as } |x|\to\infty.
\endaligned
\right.
\end{equation} 
It follows from \cite{Kenig-Shen-2} that if $p=2$ and $\Omega$ is a bounded Lipschitz domain with connected
boundary, the unique solution to (\ref{exterior-Neumann}) satisfies
the estimate $\| (\nabla u_\varep)^*\|_{L^2(\partial\Omega)}
\le C\, \| g\|_{L^2(\partial\Omega)}$
(if  $\partial\Omega$ is not connected, the data $g$ needs to satisfy some compatibility
conditions).
An careful inspection of Theorem \ref{maximal-function-theorem}
shows that the $L^2$ results extend to $L^p$ for $1<p<\infty$, 
if $\Omega$ is a bounded $C^{1, \alpha}$
domain.
}
\end{remark}

\section{$L^p$ Regularity problem}

In this section we outline the proof of the following.

\begin{thm}\label{regularity-theorem}
Suppose that $A\in \Lambda (\mu, \lambda, \tau)$ and $A^*=A$.
Let $\Omega$ be a bounded $C^{1, \alpha}$ domain with connected boundary and $1<p<\infty$.
Then, for any $f\in W^{1,p}(\partial\Omega)$, the unique solution to
$\mathcal{L}_\varep (u_\varep) =0$ in $\Omega$, $u_\varep =f$ on $\partial\Omega$
and $(\nabla u_\varep)^*\in L^p(\partial\Omega)$ satisfies the estimate
\begin{equation}\label{10.1}
\|(\nabla u_\varep)^*\|_{L^p(\partial\Omega)}
\le C\, \| \nabla_{tan} f\|_{L^p(\partial\Omega)},
\end{equation}
where $C$ depends only on $d$, $m$, $p$, $\mu$, $\lambda$, $\tau$ and $\Omega$.
\end{thm}

The case $p=2$ was proved in \cite{Kenig-Shen-2} for Lipschitz domains. 
The case $p>2$ follows from  Lemma \ref{lemma-9.1} 
and the following analog of Lemma
\ref{lemma-9.2}.

\begin{lemma}\label{lemma-10.2} 
Suppose that $A\in\Lambda (\lambda, \mu, \tau)$ and $A^*=A$.
Let $p>2$ and $\Omega$ be a bounded Lipschitz domain with connected boundary.
Assume that
\begin{equation}\label{10.2.1}
\left(\average_{B(Q,r)\cap\partial\Omega} |(\nabla u_\varep)^*|^p\, d\sigma \right)^{1/p}
\le
C_0\, \left(\average_{B(Q,2r)\cap\partial\Omega} |(\nabla u_\varep)^*|^2\, d\sigma \right)^{1/2},
\end{equation}
whenever $u_\varep \in W^{1,2}(B(Q,3r)\cap \Omega)$
is a weak solution to $\mathcal{L}_\varep (u_\varep)=0$ in $B(Q, 3r)\cap\Omega$
and $u_\varep=0$ on $B(Q,3r)\cap\partial\Omega$
for some $Q\in \partial\Omega$ and $0<r<r_0$.
Then the weak solution to $\mathcal{L}_\varep(u_\varep)=0$ in $\Omega$
and $u_\varep =f\in W^{1,p}(\partial\Omega)$
satisfies the estimate $\|(\nabla u_\varep)^*\|_{L^p(\partial\Omega)}
 \le C\, \| \nabla_{tan}f \|_{L^p(\partial\Omega)}$,
where $C$ depends only on $d$, $m$, $p$, $\mu$, $\lambda$, $\tau$, $r_0$,
$C_0$ and $\Omega$.
\end{lemma}

The proof of Lemma \ref{lemma-10.2} is similar to that of Lemma \ref{lemma-9.2}.
We refer the reader to \cite{Kilty-Shen-regularity} where a similar statement was proved
for elliptic equations with constant coefficients.

To handle the case $1<p<2$, we follow the approach for Laplace's equation in Lipschitz domains
\cite{Dahlberg-Kenig-1987} and consider $L^2$ solutions
with Dirichlet data $u_\varep =a$, where supp$(a)\subset B(Q,r)\cap\partial\Omega$
for some $Q\in \partial\Omega$ and $0<r<r_0$, and $\|\nabla_{tan} a\|_{L^\infty(\partial\Omega)}
\le r^{1-d}$. By interpolation it suffices to show estimate (\ref{9.3.1}).
Note that $|a|\le C r^{2-d}$. Using the estimates on Green's functions in \cite{AL-1987}, one has
\begin{equation}\label{10.2.2}
|\nabla u_\varep (x)|\le \frac{Cr}{|x-Q|^d} \qquad \text{ if } |x-Q|\ge Cr.
\end{equation}
Estimate (\ref{9.3.1})
follows easily from the $L^2$ estimate $\|(\nabla u_\varep)^*\|_{L^2(\partial\Omega)}
\le C\|\nabla_{tan} a\|_{L^2(\partial\Omega)}$ and (\ref{10.2.2}).

\begin{remark}\label{remark-10.1}
{\rm
One may also consider the $L^p$ regularity problem for the exterior domain:
given $f\in W^{1,p}(\partial\Omega)$, find a solution $u_\varep$ to $\mathcal{L}_\varep 
(u_\varep) =0$ in $\Omega_-$ such that $u_\varep=f$ on $\partial\Omega$,
$(\nabla u_\varep)^*\in L^p(\partial\Omega)$ and $u_\varep (x)=O(|x|^{2-d})$
as $|x|\to\infty$.
It follows from \cite{Kenig-Shen-2} that if $\Omega$ is a bounded 
Lipschitz domain in $\br^d$, $d\ge 3$, 
then the unique solution to the $L^2$ regularity problem in $\Omega_-$
satisfies the estimate $\|(\nabla u_\varep)^*\|_{L^2(\partial\Omega)}
\le C\, \| \nabla_{tan} f\|_{W^{1,2}(\partial\Omega)}$.
An inspection of Theorem \ref{regularity-theorem} shows that
the $L^2$ result extends to $L^p$ for $1<p<\infty$, 
if $\Omega$ is a $C^{1,\alpha}$ domain.
}
\end{remark}

\section{Representation by layer potentials}

For $f\in L^p(\partial\Omega)$, the single layer potential $u_\varep =\mathcal{S}_\varep(f)$ and
double layer potential $w_\varep=\mathcal{D}_\varep (f)$ for the operator
$\mathcal{L}_\varep$ in $\Omega$ are defined by
\begin{equation}\label{layer-potential}
\aligned
u_\varep^\alpha (x) &=\int_{\partial\Omega}
\Gamma_{A,\varep}^{\alpha\beta} (x,y) f^\beta (y)\, d\sigma (y),\\
w^\alpha_\varep (x) &=\int_{\partial\Omega}
\left( \frac{\partial}{\partial \nu_{\varep}^*}
\big\{ \Gamma_{A^*,\varep}^\alpha (y, x)\big\}\right)^\beta f^\beta (y)\, d\sigma (y),
\endaligned
\end{equation}
where $\Gamma_{A,\varep}(x,y)$ and $\Gamma_{A^*,\varep} (x,y)=(\Gamma_{A, \varep} (y,x))^*$ 
are the fundamental
solutions for $\mathcal{L}_\varep$ and
$(\mathcal{L}_\varep)^*$ respectively.
Both $\mathcal{S}_\varep (f)$ and $\mathcal{D}_\varep (f)$
are solutions of $\mathcal{L}_\varep (u)=0$
in $\br^d\setminus \partial\Omega$. Under the assumptions that
$A\in \Lambda (\mu, \lambda,\tau)$ and $\Omega$ is a bounded Lipschitz domain,
it was proved in \cite{Kenig-Shen-2} that for $1<p<\infty$,
$$
\| \big(\nabla \mathcal{S}_\varep (f)\big)^*\|_{L^p(\Omega)}
+\| \big( \mathcal{D}_\varep (f)\big)^*\|_{L^p(\partial\Omega)} 
\le C_p \| f\|_{L^p(\partial\Omega)},
$$
where $C_p$ depends only on $d$, $m$, $\mu$, $\lambda$, $\tau$, $p$ and
the Lipschitz character of $\Omega$.
Furthermore, $(\nabla u_\varep)_\pm (P)$ exists for a.e. $P\in \partial\Omega$,
$\left(\frac{\partial u_\varep}{\partial \nu_\varep}\right)_\pm
=(\pm \frac12 I + \mathcal{K}_{A, \varep}) (f)$ and
$ (w_\varep)_\pm = (\mp\frac12 I +\mathcal{K}_{A^*, \varep}^*) (f)$,
where $\mathcal{K}_{A^*,\varep}^* $ is the adjoint of $\mathcal{K}_{A^*, \varep}$.
Here $(u)_\pm$ denotes the nontangential limits on $\partial\Omega$ of $u$, taken 
from $\Omega$ and $\Omega_-$ respectively.

Let $L^p_0 (\partial\Omega, \br^m)$ denote the space of functions in $L^p(\partial\Omega,\br^m)$
with  mean value zero.

\begin{thm}\label{layer-potential-theorem}
Let $\Omega$ be a bounded $C^{1,\alpha}$ domain in $\br^d$, $d\ge 3$
with connected boundary.
Suppose that $A\in \Lambda(\mu, \lambda, \tau)$ and $A^*=A$.
Then, for $1<p<\infty$, 
\begin{equation}\label{11.1}
\aligned
\frac12 I +\mathcal{K}_{A, \varep} &:
L_0^p(\partial\Omega, \br^m) \to L_0^p(\partial\Omega, \br^m),\\
-\frac12 I +\mathcal{K}^*_{A^*, \varep} &:
L^p(\partial\Omega, \br^m) \to L^p(\partial\Omega, \br^m),\\
 \mathcal{S}_\varep &: L^p(\partial\Omega, \br^m)\to
W^{1,p}(\partial\Omega, \br^m),
\endaligned
\end{equation}
are invertible and the operator norms of their inverses are bounded by
a constant independent of $\varep$.
\end{thm}

\begin{proof}
The case $p=2$ was proved in \cite{Kenig-Shen-2} for Lipschitz domains.
If $\Omega$ is $C^{1,\alpha}$, the results for $p\neq 2$ follow from the solvabilities of
the $L^p$ Neumann and regularity problems with uniform estimates in $\Omega$ and $\Omega_-$
(see Theorem \ref{maximal-function-theorem}, Theorem \ref{regularity-theorem},
Remarks \ref{remark-9.1} and \ref{remark-10.1}).
\end{proof}

As a corollary, solutions to the $L^p$ Dirichlet, Neumann and regularity problems for
$\mathcal{L}_\varep (u_\varep)=0$ may be represented
by layer potentials with uniformly $L^p$ bounded density functions.
This shows that the classical method of integral equations applies to the 
elliptic system $\mathcal{L}_\varep (u_\varep)=0$.

\begin{thm}\label{representation-theorem} Let $1<p<\infty$. 
 Under the same assumptions on $A$ and $\Omega$
as in Theorem \ref{layer-potential-theorem}, the following holds.

\item

(i) For $g\in L^p(\partial\Omega)$, the solution 
to the $L^p$ Dirichlet problem in $\Omega$ with $u_\varep=g$ on $\partial\Omega$ is given
by $u_\varep =\mathcal{D}_\varep (h_\varep)$ with $\|h_\varep\|_{L^p(\partial\Omega)}
\le C_p\|g\|_{L^p(\partial\Omega)}$.

\item

(ii) For $g\in L^p(\partial\Omega)$, the solution to the $L^p$ Neumann problem in $\Omega$ 
with $\frac{\partial u_\varep}{\partial\nu_\varep} =g$ on $\partial\Omega$
is given by
$u_\varep =\mathcal{S}_\varep (h_\varep)$ with $\|h_\varep\|_{L^p(\partial\Omega)}
\le C_p\| g\|_{L^p(\partial\Omega)}$.

\item

(iii)
For $g\in W^{1,p}(\partial\Omega)$, 
the solution to the $L^p$ regularity problem in $\Omega$ with $u_\varep =g$ on $\partial\Omega$
is given by
$u_\varep =\mathcal{S}_\varep (h_\varep)$ with $\|h_\varep\|_{L^p(\partial\Omega)}
\le C_p\| g\|_{L^p(\partial\Omega)}$.

\end{thm}

\bibliography{kls1}

\small
\noindent\textsc{Department of Mathematics, 
University of Chicago, Chicago, IL 60637}\\
\emph{E-mail address}: \texttt{cek@math.uchicago.edu} \\

\noindent \textsc{Courant Institute of Mathematical Sciences, New York University, New York, NY 10012}\\
\emph{E-mail address}: \texttt{linf@cims.nyu.edu}\\

\noindent\textsc{Department of Mathematics, 
University of Kentucky, Lexington, KY 40506}\\
\emph{E-mail address}: \texttt{zshen2@email.uky.edu} \\

\noindent \today

\end{document}